\documentclass[a4paper]{amsart}
\usepackage{a4wide}
\usepackage{microtype}
\usepackage{amsmath}
\usepackage{amsthm}
\usepackage{amsfonts}
\usepackage{amssymb}
\usepackage{mathrsfs}
\usepackage{tikz}
\usepackage[alphabetic,initials,nobysame,non-compressed-cites]{amsrefs}
\usepackage[hidelinks]{hyperref}
\usepackage{graphicx}
\usepackage{overpic}
\usepackage{enumitem}
\usepackage{array}
\usepackage{pdflscape}
\usepackage{todonotes}

\theoremstyle{plain}
\newtheorem{theorem}{Theorem}[section]
\newtheorem{lemma}[theorem]{Lemma}
\newtheorem{corollary}[theorem]{Corollary}

\newtheorem{proposition}[theorem]{Proposition}

\theoremstyle{definition}
\newtheorem{definition}[theorem]{Definition}

\theoremstyle{remark}
\newtheorem*{remark}{Remark}

\def \PP {\mathbb P}
\def \RR {\mathbb R}
\def \cB {\mathcal B}
\def \cP {\mathcal P}
\def \cC {\mathcal C}
\def \cT {\mathcal T}
\def \cI {\mathcal I}
\def \sC {\mathscr C}
\def \bu {\mathbf u}
\def \bv {\mathbf v}
\def \bx {\mathbf x}
\def \by {\mathbf y}
\def \be {\mathbf e}
\def \bt {\mathbf t}
\def \bh {\mathbf h}
\def \proj #1{\widehat{#1}}
\def \normal #1{\overline{#1}}

\DeclareMathOperator{\Span}{Span}
\DeclareMathOperator{\Cone}{Cone}
\DeclareMathOperator{\Affine}{Aff}
\DeclareMathOperator{\Interior}{Int}
\DeclareMathOperator{\Convex}{Conv}
\DeclareMathOperator{\Pyramid}{Pyr}
\DeclareMathOperator{\length}{\ell}
\DeclareMathOperator{\ball}{\mathsf{Ball}}

\title[More Infinite Ball Packings]{Even More Infinite Ball Packings\\from Lorentzian Root Systems}
\author{Hao Chen}
\address{Department of Mathematics and Computer Science, Technische Universiteit Eindhoven}
\email{hao.chen@tue.nl}


\thanks{This research was supported by the Deutsche Forschungsgemeinschaft
within the Research Training Group ``Methods for Discrete Structures'' (GRK
1408) and by the ERC Advanced Grant number 247029 ``SDModels'' while the author
was at Freie Universit\"at Berlin.}

\begin{document}

\begin{abstract}
	Boyd (1974) proposed a class of infinite ball packings that are generated by
	inversions.  Later, Maxwell (1983) interpreted Boyd's construction in terms
	of root systems in Lorentz space.  In particular, he showed that the
	space-like weight vectors correspond to a ball packing if and only if the
	associated Coxeter graph is of ``level~$2$''.   In Maxwell's work, the simple
	roots form a basis of the representations space of the Coxeter group.  In
	several recent studies, the more general based root system is considered,
	where the simple roots are only required to be positively independent.  In
	this paper, we propose a geometric version of ``level'' for the root system
	to replace Maxwell's graph theoretical ``level''.  Then we show that
	Maxwell's results naturally extend to the more general root systems with
	positively independent simple roots.  In particular, the space-like extreme
	rays of the Tits cone correspond to a ball packing if and only if the root
	system is of level $2$.  We also present a partial classification of
	level-$2$ root systems, namely the Coxeter $d$-polytopes of level-$2$ with
	$d+2$ facets.
\end{abstract}

\maketitle

\section{Introduction}\label{sec:Intro}
The title refers to a paper of Boyd titled ``A new class of infinite sphere
packings''~\cite{boyd1974}, in which he described a class of infinite ball
packings that are generated by inversions, generalising the famous Apollonian
disk packing.

Maxwell~\cite{maxwell1982} generalized Boyd's construction by interpreting the
ball packing as the space-like weights of an infinite root system in Lorentz
space.  In particular, Maxwell defined the ``level'' of a Coxeter graph as the
smallest integer $l$ such that the deletion of any $l$ vertices leaves a
Coxeter graph for finite or affine Coxeter system.  He then proved that the
space-like weights correspond to a ball packing if and only if the associated
Coxeter graph is of level $2$.


Labb\'e and the author~\cite{chen2014} revisited Maxwell's work, and found
connections with recent works on limit roots.  \emph{Limit roots} are
accumulation points of the roots in projective space.  The notion was
introduced and studied in~\cite{hohlweg2014}, where the authors also proved
that limit roots lie on the isotropic cone of the quadratic space.  The
relations between limit roots and the imaginary cone are investigated
in~\cite{dyer2013} and~\cite{dyer2013b}.

For root systems in Lorentz space, the set of limit roots is equal to the limit
set of the Coxeter group seen as a Kleinian group acting on the hyperbolic
space~\cite{hohlweg2013}.  In~\cite{chen2014}, we proved that the accumulation
points of the roots and of the weights coincide on the light cone in the
projective space.  As a consequence, when the Coxeter graph is of level $2$,
the set of limit roots is the residue set of the ball packing described by Boyd
and Maxwell.  Furthermore, we gave a geometric interpretation for Maxwell's
notion of level, described the tangency graph of the Boyd--Maxwell ball packing
in terms of the Coxeter complex, and completed the enumeration of $326$ Coxeter
graphs of level $2$.

Comparing to~\cite{maxwell1982}, the root systems considered in most studies of
limit roots (e.g.~\cite{dyer2013b, hohlweg2014}) are more general in several
ways:

First, the root systems considered in~\cite{dyer2013b, hohlweg2014} are not
necessarily in a Lorentz space.  For non-Lorentzian root systems, we
conjectured in~\cite{chen2014} that the accumulation points of roots still
coincide with the accumulation points of weights.

Second, even if the root system is Lorentzian, the associated Coxeter graph is
not necessarily of level $2$.  The cases of level $\ne 2$ were also
investigated in~\cite{chen2014}.  It turns out that no ball appears if the
Coxeter graph is of level $1$, and balls may intersect if the Coxeter graph is
of level $>2$.  In either case, it remains true that the set of limit roots is
the residue set of the balls corresponding to the space-like weights.

The current paper deals with a third gap.  Maxwell only considered the case
where the simple roots form a basis of the representation space, so one can
define the fundamental weights as the dual basis.  However, in~\cite{dyer2013b}
and~\cite{hohlweg2014}, the more general based root system is considered, which
only requires the simple roots to be positively independent, but not
necessarily linearly independent.

In order to extend Maxwell's results to based root systems, we propose the
notion of ``level'' for root systems to replace Maxwell's graph theoretical
``level''.  The definition is based on the geometric interpretation
in~\cite{chen2014}.  When the simple roots are linearly independent as
in~\cite{maxwell1982}, our level for the root system and Maxwell's level for
the Coxeter graph coincide.  Moreover, in place of the weight vectors, we will
look at the extreme rays of the Tits cone.  Then we show in
Section~\ref{sec:Maxwell} that all the results in~\cite{maxwell1982}
and~\cite{chen2014} extend to the more general setting.  In particular:

\begin{theorem}[extending \cite{maxwell1982}*{Theorem 3.2}]\label{thm:packing}
	The space-like extreme rays of the Tits cone correspond to a ball packing if
	and only if the based root system is of level~$2$.
\end{theorem}

The correspondance will be explained in Section~\ref{ssec:packing}.  We
consider the ball packing in the theorem as associated to the based root
system, then

\begin{theorem}[{extending \cite{chen2014}*{Theorem 3.8}}]\label{thm:limroots}
	The set of limit roots of a Lorentzian root system of level~$2$ is equal to
	the residual set of the associated ball packing.
\end{theorem}

Maxwell's proofs rely on the decomposition of vectors into basis vectors (i.e.
the simple roots), which is not possible in our setting.  Hence many proofs
need to be revised.  Our proofs will make heavy use of projective geometry.

For many Lorentzian root systems of level $2$, the associated Coxeter graph is
not of level $2$ in Maxwell's sense.  So our results imply many new infinite
ball packings generated by inversions.  In Section~\ref{sec:Tumarkin}, we
provide a partial classification of Lorentzian root systems of level $2$.  More
specifically, we try to classify the Coxeter $d$-polytopes of level $2$ with
$d+2$ facets.  For this, we follow the approach of~\cite{kaplinskaja1974,
esselmann1996, tumarkin2004} for enumerating hyperbolic Coxeter $d$-polytopes
with $d+2$ facets, and take advantage of previous enumerations of Coxeter
systems, such as~\cite{lanner1950, chein1969, chen2014}.  Our enumeration makes
contribution to the study of infinite-covolume hyperbolic reflection groups.

\section{Level of Lorentzian root systems}\label{sec:Coxeter}
\subsection{Lorentz space}

A \emph{quadratic space} is a pair $(V,\cB)$ where $V$ is a real vector space
and~$\cB$ is a symmetric bilinear form on $V$.  Two vectors $\bx,\by\in V$ are
said to be \emph{orthogonal} if $\cB(\bx,\by)=0$. For a subspace $U\subseteq
V$, its \emph{orthogonal companion} is the set
\[
	U^\perp=\{\bx\in V\mid\cB(\bx,\by)=0\text{ for all }\by\in U\}.
\]
The orthogonal companion $V^\perp$ of the whole space $V$ is called the
\emph{radical}.  We say that $(V,\cB)$ is \emph{degenerate} if the radical
$V^\perp$ contains non-zero vectors.  In this case, the matrix of
$B=(\cB(\be_i,\be_j))$ is singular for any basis $\{\be_i\}$ of $V$.

The signature of $(V,\cB)$ is the triple $(n_+, n_0, n_-)$ indicating the
number of positive, zero and negative eigenvalues of the matrix $B$.  For
non-degenerate spaces we have $n_0=0$.  A non-degenerate space is an
\emph{Euclidean space} if $n_-=0$ (i.e.\ $B$ is positive definite), or a
\emph{Lorentz space} if $n_-=1$.

The group of linear transformations of $V$ that preserve the bilinear form
$\cB$ is called an \emph{orthogonal group}, and is denoted by $O_\cB(V)$.  The
orthogonal group of a Lorentz space is called a \emph{Lorentz group}.

The set
\[
	Q=\{\bx\in V\mid\cB(\bx,\bx)=0\}
\]
is called the \emph{isotropic cone}, and vectors in $Q$ are said to be
\emph{isotropic}.  In Lorentz space, the isotropic cone is called the
\emph{light cone}, and isotropic vectors are said to be \emph{light-like}.  Two
light-like vectors are orthogonal if and only if one is the scalar multiple of
the other.  A non-isotropic vector $\bx\in V$ is said to be \emph{space-like}
(resp.~\emph{time-like}) if $\cB(\bx,\bx)>0$ (resp.~$<0$).  A subspace
$U\subseteq V$ is said to be \emph{space-like} if its non-zero vectors are all
space-like, \emph{light-like} if it contains some non-zero light-like vector
but no time-like vector, or \emph{time-like} if it contains time-like vectors.  

For a non-isotropic vector $\alpha\in V$, the \emph{reflection} in $\alpha$ is
defined as the map
\[
	s_\alpha(\bx) = \bx - 2 \frac{\cB(\alpha,\bx)}{\cB(\alpha,\alpha)} \alpha,
	\quad \text{for all } \bx \in V.
\]
The orthogonal companion of $\alpha\RR$,
\[
	H_\alpha=\{\bx\in V\mid\cB(\bx,\alpha)=0\},
\]
is the hyperplane fixed by the reflection $s_\alpha$, and is called the
\emph{reflecting hyperplane} of $\alpha$.  One verifies that $\alpha$ is
space-like (resp.\ time-like) if and only if $H_\alpha$ is time-like
(resp.\ space-like).  

\subsection{Based root systems}

The based root system has been the framework of several recent studies of
infinite Coxeter systems, including \cites{howlett1997, krammer2009, dyer2013b,
hohlweg2014} etc., and traces back to Vinberg \cite{vinberg1971}.  

Recall that an abstract \emph{Coxeter system} is a pair $(W,S)$, where $S$ is a
finite set of generators and the \emph{Coxeter group}~$W$ is generated by $S$
with the relations~$(st)^{m_{st}}=e$ where $s,t\in S$, $m_{ss}=1$ and
$m_{st}=m_{ts}\geq 2$ or $=\infty$ if $s\neq t$. The cardinality $n=|S|$ is the
\emph{rank} of the Coxeter system $(W,S)$. For an element $w\in W$, the
\emph{length} of $w$, denoted by ~$\length(w)$, is the smallest natural number
$k$ such that $w=s_1s_2\dots s_k$ for $s_i\in S$.  Readers unfamiliar with
Coxeter groups are invited to consult \cite{bourbaki1968, humphreys1992} for
basics.

Let $(V,\cB)$ be a quadratic space. A \emph{root basis} $\Delta$ in
$(V,\cB)$ is a \emph{finite} set of vectors in $V$ such that
\begin{enumerate}
	\item $\cB(\alpha,\alpha)=1$ for all $\alpha\in\Delta$;

	\item $\cB(\alpha,\beta) \in ]-\infty,-1] \cup
		\{-\cos(\pi/k),\,k\in\mathbb{Z}_{\ge 2}\}$ for all
		$\alpha\ne\beta\in\Delta$;

	\item $\Delta$ is \emph{positively independent}.  That is, a linear
		combination of $\Delta$ with non-negative coefficient only vanishes when
		all the coefficients vanish.
\end{enumerate}
We assume that $\Delta$ spans $V$.  If this is not the case, we replace $V$ by
the subspace $\Span(\Delta)$, and $\cB$ by its restriction on $\Span(\Delta)$.

Following~\cite{howlett1997}, we call $\Delta$ a \emph{free root basis} if it
also forms a basis of $V$.  Following~\cite{krammer2009}, we call a free root
basis $\Delta$ \emph{classical} if $\cB(\alpha, \beta) \ge -1$ for all $\alpha,
\beta \in \Delta$.  The classical root basis is presented in textbooks such
as~\cite{bourbaki1968, humphreys1992}.

Let $S=\{s_\alpha\mid\alpha\in\Delta\}$ be the set of reflections in vectors of
$\Delta$, and $W$ be the reflection subgroup of $O_\cB(V)$ generated by $S$.
Then $(W,S)$ is a \emph{Coxeter system}, where the order of $s_\alpha s_\beta$
is $k$ if $\cB(\alpha,\beta)=-\cos(\pi/k)$, or $\infty$ if
$\cB(\alpha,\beta)\le -1$.  Let $\Phi:=W(\Delta)$ be the orbit of $\Delta$
under the action of $W$, then the pair $(\Delta,\Phi)$ is called a \emph{based
root system in $(V, \cB)$ with associated Coxeter system $(W,S)$}.  Vectors in
$\Delta$ are called \emph{simple roots}, and vectors in $\Phi$ are called
\emph{roots}.  The roots~$\Phi$ are partitioned into \emph{positive roots}
$\Phi^+ = \Cone(\Delta) \cap \Phi$ and \emph{negative roots} $\Phi^-=-\Phi^+$.
The \emph{rank} of $(\Delta,\Phi)$ is the cardinality of $\Delta$.  In the
following, we write~$w(\bx)$ for the action of~$w\in W$ on $V$, and the word
``based'' is often omitted.

A based root system in Euclidean space has a classical root basis, in which
case the associated Coxeter system is of \emph{finite type}, and we say the
root system is \emph{finite}.  If $(V,\cB)$ is degenerate yet $\cB(\bv,\bv)\ge
0$ for all $\bv\in V$, the based root system $(\Delta,\Phi)$ in $(V,\cB)$ also
has a classical root basis; in this case, we say that $(\Delta,\Phi)$ is
\emph{affine} since the associated Coxeter system must be of affine type.  If
$(V,\cB)$ is a Lorentz space, we say that the root system $(\Delta,\Phi)$ is
\emph{Lorentzian}.  Finally, $(\Delta,\Phi)$ is said to be non-degenerate if
$(V,\cB)$ is.

A based root system is \emph{irreducible} if there is no proper partition
$\Delta = \Delta_I \sqcup \Delta_J$ such that $\cB(\alpha,\beta)=0$ for all
$\alpha \in \Delta_I$ and $\beta \in \Delta_J$, in which case the associated
Coxeter group is also irreducible.

\subsection{Geometric representations of Coxeter systems}

Given an abstract Coxeter system $(W,S)$ of rank $n$.  We introduce a matrix
$B$ such that
\[
  B_{st}=
	\begin{cases}
		-\cos(\pi/m_{st}) & \text{if}\quad m_{st}<\infty,\\
		-c_{st} & \text{if}\quad m_{st}=\infty,
	\end{cases}
\]
for $s,t\in S$, where $c_{st}$ are chosen arbitrarily with $c_{st} = c_{ts}
\geq 1$.  Note that if $m_{st}$ is finite for all $s,t\in S$, there is only one
choice for the matrix $B$.  This is the case when, for instance, $(W,S)$ is of
finite or affine type (with the exception of $I_\infty$).  In~\cite{chen2014},
the Coxeter system~$(W,S)$ associated with the matrix~$B$ is referred to as a
\emph{geometric Coxeter system}, and denoted by~$(W,S)_B$.

With the matrix $B$, there is a canonical way to associate $(W,S)$ to a root
system with free root basis.  Let $V$ be a real vector space of dimension $n$
with basis $\{\be_s\}_{s\in S}$ indexed by the elements in $S$.  Then the
matrix~$B$ defines a bilinear form $\cB$ on $V$ by $\cB(\be_s, \be_t) =
\be_s^\intercal B \be_t$ for $s,t\in S$.  The basis $\{\be_s\}_{s\in S}$ form a
free root basis in $(V,\cB)$.  The homomorphism that maps $s\in S$ to the
reflection in $\be_s$ is a faithful \emph{geometric representation} of the
Coxeter group $W$ as a discrete reflection subgroup of the orthogonal
group~$O_\cB(V)$.  This is the representation considered in
Maxwell~\cite{maxwell1982}.

In the present paper, we focus on based root systems in Lorentz space, which is
non-degenerate.  Moreover, we are particularly interested in root basis that is
not linearly independent.  Hence we prefer non-degenerate root systems.  With
the free root basis above, there is a canonical way to obtain a non-degenerate
root system by ``dividing out the radical''~\cite[Section~6.1]{krammer2009}.

If the matrix $B$ is singular with rank $d$, then the dimension of the
\emph{radical} $V^\perp$ is $n-d$.  The bilinear form $\cB$ restricted on the
quotient space $U=V/V^\perp$ is non-degenerate.  Let $\alpha_s$ be the
projection of $\be_s$ onto $U$ for all $s\in S$.  If the root system induced by
$\{\be_s\}_{s\in S}$ is not affine, the vectors $\Delta=\{\alpha_s\mid s\in
S\}$ are positively independent and form a root basis in $(U,\cB)$ such that
$\cB(\alpha_s,\alpha_t)=B_{st}$~\cite{krammer2009}*{Proposition 6.1.2}.  The
homomorphism that maps $s\in S$ to the reflection in~$\alpha_s$ is a faithful
\emph{geometric representation} of the Coxeter group $W$ as a discrete
reflection subgroup of the orthogonal group $O_\cB(U)$.

This process does not work for affine root systems.  We call a root system
\emph{canonical} if it is non-degenerate or affine.  The process described in
the previous paragraph is called \emph{canonicalization}.  Unless otherwise
stated, root systems in this paper are all canonical.

We adopt Vinberg's convention to encode the matrix $B$ into the Coxeter graph.
That is, if $c_{st}> 1$ the edge~$st$ is dotted and labeled by~$-c_{st}$. This
convention is also used by Abramenko--Brown
\cite[Section~10.3.3]{abramenko2008} and Maxwell \cite[Section~1]{maxwell1982}.
The Coxeter graph is connected if and only if the Coxeter system it represents
is irreducible.

Maxwell proposed the following definition:
\begin{definition}[Level of a Coxeter graph]
	The \emph{level of a Coxeter graph} is the smallest integer $l$ such that
	deletion of any $l$ vertices leaves a Coxeter graph of a finite or affine
	root system.
\end{definition}

\subsection{Facials subsets and the level of root systems}

The readers are assumed to be familiar with convex cones.  Otherwise we
recommend~\cite{rockafellar1970} or the appendix of~\cite{dyer2013} for
reference.

Let $(\Delta, \Phi)$ be a canonical root basis in $(V,\cB)$ with associated
Coxeter system $(W,S)$.  Because of the positive independence, the cone
$\cC=\Cone(\Delta)$ is a pointed polyhedral cone.  We call $\cC$ the
\emph{positive cone} since it is also spanned by the positive roots, i.e.\
$\cC=\Cone(\Phi^+)$.  The extreme rays of $\cC$ are spanned by the simple roots
in $\Delta$.

A subset $\Delta'$ of $\Delta$ is said to be \emph{$k$-facial}, $1\le k\le d$,
if $\Cone(\Delta')$ is a face of codimension $k$ of $\cC$.  In this case, we
say that $I = \{ s_\alpha \mid \alpha \in \Delta' \} \subset S$ is a
\emph{$k$-facial} subset of $S$, and use the notation $\Delta_I$ in place of
$\Delta'$.  We also write $W_I=\langle I \rangle$, $V_I = \Span(\Delta_I)$ and
$\Phi_I = W_I(\Delta_I)$.  Then a $k$-facial subset $I$ of $S$ induce a based
root system $(\Delta_I, \Phi_I)$ in $(V_I,\cB)$ with associated Coxeter system
$(W_I, I)$.  We call $(\Delta_I, \Phi_I)$ a $k$-facial root subsystem.  We
agree on the convention that $\Delta$ and $S$ themselves are $0$-facial.  

We now define the central notion of this paper.
\begin{definition}[Level of a root system]\label{def:level}
	Let $(\Delta,\Phi)$ be a root system in $(V,\cB)$ with canonical root basis
	$\Delta$ and associated Coxeter system $(W,S)$.  The \emph{level} of
	$(\Delta,\Phi)$ is the smallest integer $l$ such that every $l$-facial
	subsystem is (after canonicalization if necessary) finite or affine.
\end{definition}

So root systems of finite or affine type are of level $0$.  If a root system is
of level $l$ and every $l$-facial subsystem is of finite type, we say
that it is \emph{strictly} of level $l$.

\begin{remark}
  We only defined the notion of ``facial'' and ``level'' for canonical root
  basis.  However, it is possible that a facial root subsystem is not
  canonical.  This does not happen for Lorentzian root systems, whose facial
  root subsystems are either finite, affine or Lorentzian.  Hence in this
  paper, the canonicalization in the definition of ``level'' is never needed.
  For non-Lorentzian root systems, we propose for the moment to first
  canonicalize a facial root subsystem if necessary.  This is however open for
  future discussions.
\end{remark}

For free root basis, the level of the root system equals the level of the
Coxeter graph.  Otherwise, the two levels are different in general.  We will
show that Definition~\ref{def:level} is more suitable for geometric studies.

Given a root system of level $l$, the bilinear form $\cB$ is positive
semidefinite on $\Span(\Delta_I)$ for every $l$-facial $I \subset S$, and
indefinite for at least one $(l-1)$-facial $I \subset S$.  If the root
system $(\Delta,\Phi)$ is Lorentzian, we can reformulate
Definition~\ref{def:level} in terms of its positive cone $\cC$: The level of
$(\Delta,\Phi)$ is $1+$ the maximum codimension of the time-like faces of
$\cC$.  Here, we use the conventions that a face of codimension $0$ is the
polytope itself.


\subsection{Fundamental cone, Tits cone and imaginary cone}

In this part, we assume that the root system is non-degenerate, so that we can
identify the dual space $V^*$ with $V$.  For definitions for degenerate root
systems, see the remark in~\cite{dyer2013}*{\S 1.9}.

The dual cone $\cC^*$ of the positive cone is called the \emph{fundamental
chamber}.  Recall that
\[
	\cC^*=\{\bx \in V \mid \cB(\bx,\by) \ge 0 \text{ for all } \by \in \cC\}.
\]
The fundamental chamber is the fundamental domain for the action of $W$ on $V$.

Since $(V,\cB)$ is non-degenerate, $\cC^*$ is also a pointed polyhedral cone.
The supporting hyperplanes at the facets of $\cC^*$ are the reflecting
hyperplanes of the simple roots.  So the Coxeter group $W$ is isomorphic to the
group generated by reflections in the facets of $\Cone(\Delta^*)$, and the
stabilizer of a face of $\Cone(\Delta^*)$ is generated by reflections in the
facets that contains this face.

The extreme rays of $\cC^*$ correspond to the facets of $\cC$.  For each $1$-facial
subset $I \subset S$, we define the vector $\omega_I$ by
\[
	\cB(\alpha,\omega_I)=\begin{cases}
		=0, & \alpha\in\Delta_I\\
		>0, & \alpha\notin\Delta_I
	\end{cases}
\]
and 
\[
	\min_{\alpha\notin\Delta_I}\cB(\alpha,\omega_I)=1.
\]
So $\omega_I$ is a representative on the extremal ray of $\cC^*$ associated to
the facet $\Cone(\Delta_I)$ of $\cC$.  We write $\Delta^*:=\{ \omega_I \mid I
\text{ is a $1$-facial subset of } S \}$.  If $\Delta$ is a free root basis,
then $\Delta^*$ is the dual basis, i.e.\ the set of \emph{fundamental
weights}.  We use $\Omega$ to denote the orbit $W(\Delta^*)$ under the action
of $W$.

We define the Tits cone as the union of the orbits of $\cC^*$ under the action of
$W$, i.e.
\[
	\cT = \bigcup_{w\in W} w(\cC^*).
\]
It is equal to $\Cone(\Omega)$.  For finite root systems, the Tits cone is the
entire representation space $V$~\cite{abramenko2008}*{\S~2.6.3}.  Otherwise,
for infinite non-affine root systems, $\cT$ is a pointed cone.
See~\cite[Lemma~1.10]{dyer2013} for more properties of Tits cone.

For Lorentzian root systems, the dual of the closure $\overline \cT$ is the
closure of the \emph{imaginary cone} $\cI$, which is equal to the intersection
of the orbits of $\cC$ under the action of $W$~\cite[\S~3.1 and
Theorem~5.1(a,b)]{dyer2013}, i.e.
\[
	\overline \cI=\bigcap_{w\in W} w(\cC).
\]

\subsection{Projective picture and limit roots}\label{ssec:proj}

As observed in~\cite[Section~2.1]{hohlweg2014}, the roots $\Phi$ have no
accumulation point.  To study the asymptotic behavior of roots, we pass to the
projective representation space $\PP V$, i.e.\ the space of $1$-subspaces
of~$V$; see~\cite[Remark~2.2]{hohlweg2014}.  For a non-zero vector $\bx\in V$,
we denote by $\proj\bx\in\PP V$ the $1$-subspace spanned by~$\bx$.  The
geometric representation then induces a \emph{projective representation}
\[
	w\cdot\proj\bx=\proj{w(\bx)},\quad w\in W,\quad\bx\in V.
\]

For a set $X\subset V$, we have the corresponding projective set
\[
	\proj X:=\{\proj\bx\in\PP V\mid\bx\in X\}
\]
In this sense, we have the projective simple roots $\proj\Delta$, projective
roots $\proj\Phi$ and the projective isotropic cone $\proj Q$.  We use
$\Convex(\proj X)$ and $\Affine(\proj X)$ to denote $\proj{\Cone(X)}$ and
$\proj{\Span(X)}$ respectively.  

The \emph{limit roots} are the accumulation points of $\proj\Phi\subset\PP V$.
In~\cite{hohlweg2014}, it was proved that the limit roots lies on $\proj Q$.
For free root basis, it was proved in~\cite{chen2014} that limit roots are also
the accumulation points of $\proj\Omega$.  Limit roots are the projectivization
of the extreme rays of the imaginary cone~\cite{dyer2013, dyer2013b}.

The projective space $\PP V$ can be identified with an affine subspace plus a
\emph{hyperplane at infinity}.  We usually fix a vector $\bv$ and take the
affine subspace
\[
	H^1_\bv=\{\bx \in V \mid \cB(\bv,\bx)=1\}.
\]
Then for a vector $\bx\in V$, we represent $\proj\bx\in\PP V$ by the vector
$\bx/\cB(\bv,\bx)\in H^1_\bv$ if $\cB(\bv,\bx)\ne 0$, or some point at infinity
if $\cB(\bv,\bx)=0$.

For Lorentzian root systems, two choices of $\bv$ turn out to be very
convenient.  One is
\[
	\bt = -\sum_{\alpha \in \Delta} \alpha.
\]
One verifies from the definition of based root system that $\bt$ is time-like.
So the subspace $H_\bt=\bt^\perp$ is space-like and divides the space into two
parts, each containing half of the light cone.  Vectors on the same side as
$\bt$ are said to be \emph{past-directed}; they have negative inner products
with $\bt$.  Those on the other side are said to be \emph{future directed}.  It
then makes sense to call $\bt$ the \emph{direction of past}.  It turns out that
the fundamental chamber is past-directed.  Hence on the affine hyperplane
$H^1_\bt$, the light cone $\proj Q$ appears as a closed surface projectively
equivalent to a sphere, and the fundamental chamber appears as a bounded
polytope.  We call $\proj \cC^*$ the Coxeter polytope.  We can view the
interior of $\proj Q$ as the Klein model for the hyperbolic space.  Then the
study on Lorentzian root systems can be seen as a study of hyperbolic Coxeter
polytopes.

The other is
\[
	\bh = -\sum_{\omega \in \Delta^*} \omega,
\]
One verifies that $\cB(\bh,\alpha)>0$ for all the simple roots
$\alpha\in\Phi^+$, so the affine hyperplane $H_\bh^1$ is transverse for the
positive cone $\cC$.  This is useful for the visualization of the root system,
see~\cite{hohlweg2014}*{\S~5.2}, and also for a technical reason.  An important
technical notion in~\cite{chen2014} is the height $h(\bx)$ for a vector $\bx\in
V$; see~\cite[Theorem 3.2]{chen2014} and~\cite[Theorem 2.7]{hohlweg2014}.  For
non-degenerate root systems, we define $h(\bx)=\cB(\bh,\bx)$.  One verifies
that $h(\bx)$ is a $L_1$-norm on the positive roots $\Phi^+$.  If the root
basis is free, our definition coincides with the definition in~\cite{chen2014},
where $h(\bx)$ is the sum of the coordinates with simple roots as basis.

\section{Extending Maxwell's results}\label{sec:Maxwell}
In this part, we extend one by one the major results of \cite{maxwell1982} to
canonical root systems.  Detailed proofs are given only if there is a
significant difference from Maxwell's proof.  For the statement of the results,
we mimic intentionally the formulations in \cite{maxwell1982}.  Remember that
the level of a root system is in general different from the level of its
Coxeter graph.

First of all, the following two results are proved in \cite{maxwell1982} for
free root basis, and are extended to canonical root basis in
\cite{howlett1997}; see also \cite{dyer2013}*{\S 9.4}.

\begin{proposition}[\cite{howlett1997}*{Proposition 3.4}, extending \cite{maxwell1982}*{Proposition 1.2}]
	For every vector~$\bx$ in the dual of the Tits cone, $\cB(\bx,\bx)\le 0$.
\end{proposition}

\begin{corollary}[\cite{howlett1997}*{Proposition 3.7}, extending \cite{maxwell1982}*{Corollary 1.3}] 
	The Tits cone of a Lorentzian root system contains one component of the
	light cone $Q\setminus\{0\}$.
\end{corollary}

We will need the following lemma:

\begin{lemma}\label{lem:radon}
	Let $\cP$ be a polytope of dimension $d\ge 3$, and $x$ be a point in the
	exterior of $\cP$.  Then there is a line $L$ that passes through $x$ and two
	points $u \in F$ and $v \in G$, where $F$ and $G$ are two disjoint faces of
	$\cP$.  Moreover, the following are equivalent:
	\begin{enumerate}[label=\emph{(\roman*)}]
		\item For any $L$ with the property, either $u$ or $v$ is in the interior
			of a facet.

		\item $x$ is not on the affine hull of any facet, and is either beyond all
			but one facet, or beneath all but one facet.
	\end{enumerate}
\end{lemma}

Here, we say that $x$ is beyond (resp.\ beneath) a facet $F$ of $\cP$ if $x$ is
on the opposite (resp.\ same) side of $\Affine F$ as the interior of $\cP$;
see~\cite{grunbaum2003}.

\begin{proof}

	Let $H$ be any hyperplane that separates $x$ and $\cP$.  For any point $w$ on
	the boundary of $\cP$, let $\pi(w)$ be the projection of $w$ on $H$ from $x$,
	i.e.\ $\pi(w)$ is the intersection point of $H$ with the segment $[x,w]$.  So
	$\pi(\cP)$ models the polytope $\cP$ seen from $x$; see
	Figure~\ref{fig:project} for an example.  If the segment $[x,w]$ is disjoint
	from the interior of $\cP$, we say that $w$ is \emph{visible} (from $x$).  In
	Figure~\ref{fig:project}, visible edges are solid while invisible edges are
	dashed.
	
	\begin{figure}[htb]
		\centering
		\includegraphics[width=.4\textwidth]{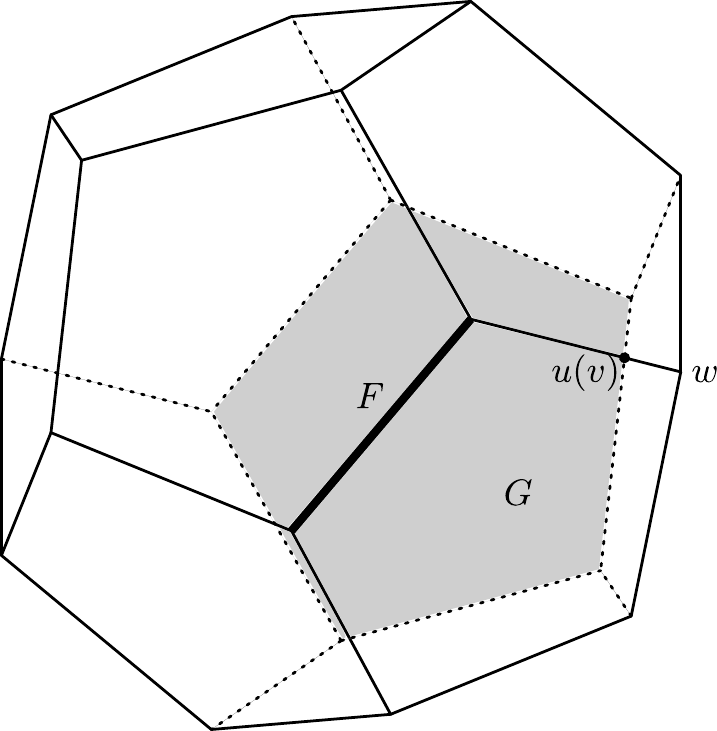}
		\caption{\label{fig:project}}
	\end{figure}

	The existence of $L$ with the given property follows from a generalization of Radon's
	theorem~\cite{bajmoczy1980}, which guarantees two disjoint faces, $F$ and
	$G$, of $\cP$ such that $\pi(F) \cap \pi(G) \ne \emptyset$.  Then $L$ is
	given by any point in the intersection.	For the equivalence, we only prove
	that (i) implies (ii) by contraposition.  The other direction is obvious.
	
	If $x \in \Affine F$ for some facet $F$ of $\cP$, we apply the generalized
	Radon's theorem to $F$, and conclude that there are two disjoint faces of $F$
	whose images under $\pi$ intersect. Any point in the intersection gives a
	line violating (i).

  Assume that $x$ is beneath at least two facets and beyond at least two other
  facets, i.e.\ there are at least two visible facets and two invisible facets.
  Let $F$ and $G$ be two disjoint faces predicted by the generalized Radon's
  theorem, where $F$ is visible while $G$ is not.  By going to the boundary if
  necessary, we may assume that one of them, say $F$, is of dimension $<d-1$.
  If $G$ is also of dimension $<d-1$, any point in the intersection gives a
  line violating (i), so $G$ must be a facet.  We further assume $\pi(F)
  \subset \Interior \pi(G)$, otherwise $\pi(F)$ intersects the boundary of
  $\pi(G)$ and any point in the intersection gives a line violating (i).  In
  Figure~\ref{fig:project}, $F$ is the thick edge, and $G$ is the gray face.

  Consider the visible faces of dimension $<d-1$ that are disjoint from $G$.
  In Figure~\ref{fig:project}, they happen to be the solid edges and their
  vertices.  The projection of their union is a connected set (because $d \ge
  3$) that contains $\pi(F) \subset \Interior\pi(G)$.  Since $G$ is not the
  only invisible facet, there is a visible vertex $w$ on the boundary of
  $\pi(\cP)$ such that $\pi(w) \not\in \pi(G)$.  So the projection of the union
  intersects the boundary of $\pi(G)$.  Any point in the intersection gives a
  line violating (i).

  To conclude, whenever $x$ does not satisfy (ii), we are able to find a line
  violating (i), which finishes the proof.
\end{proof}

\subsection{Lorentzian root systems of level~$1$ or $2$}

Two vectors $\bx,\by$ in $(V,\cB)$ are said to be \emph{disjoint} if
$\cB(\bx,\by) \le 0$ and $\cB$ is \emph{not} positive definite on the subspace
$\Span\{\bx,\by\}$.  

The proofs of the following results are apparently very different from
Maxwell's argument~\cite{maxwell1982}.  Indeed, while Maxwell's proofs make
heavy use of basis, our proofs rely primarily on projective geometry.  However,
the basic idea is the same.  In the case of free root basis, our proofs are
just geometric interpretations of Maxwell's proofs.

\begin{proposition}[extending \cite{maxwell1982}*{Proposition 1.4}]\label{prop:level1}
	If $(\Delta, \Phi)$ is level~$1$, then it is Lorentzian, and vectors in
	$\Delta^*$ are pairwise disjoint while none are space-like.
\end{proposition}

\begin{proof}
	Let $(\Delta,\Phi)$ be a canonical root system of level $1$ in $(V,\cB)$ with
	associated Coxeter system $(W,S)$.
	
	For a vector~$\bx$ such that $\cB(\bx,\bx)\le 0$, we claim that $\proj\bx \in
	\proj\cC$.  Assume without loss of generality that $h(\bx)\ge 0$.  If the
	claim is not true, we can find two vectors $\bx_+ \in \Cone(\Delta_I)$ and
	$\bx_- \in \Cone(\Delta_J)$ such that $\bx = \bx_+ - \bx_-$ for two disjoint
	facial subset $I, J \subset S$; see Figure~\ref{fig:xLvl1} (left) for the
	projective picture.  Since $\cB(\bx,\bx) = \cB(\bx_+,\bx_+) +
	\cB(\bx_-,\bx_-) - 2\cB(\bx_+,\bx_-)\le 0$, and $\cB(\bx_+,\bx_-)<0$ because
	$I$ and $J$ are disjoint, we have either $\cB(\bx_+,\bx_+)<0$ or
	$\cB(\bx_-,\bx_-)<0$, both contradict the fact that $(\Delta,\Phi)$ is of
	level $1$.  Our claim is then proved.  If $\cB(\bx,\bx)<0$, since~$\cB$ is
	positive semidefinite on the facets, $\bx$ must be in the interior of $\cC$.
	If $\cB(\bx,\bx)=0$, it is possible that $\bx$ is in the interior of a facet
	of $\cC$.

	\begin{figure}[htb]
		\centering
		\includegraphics[width=.6\textwidth]{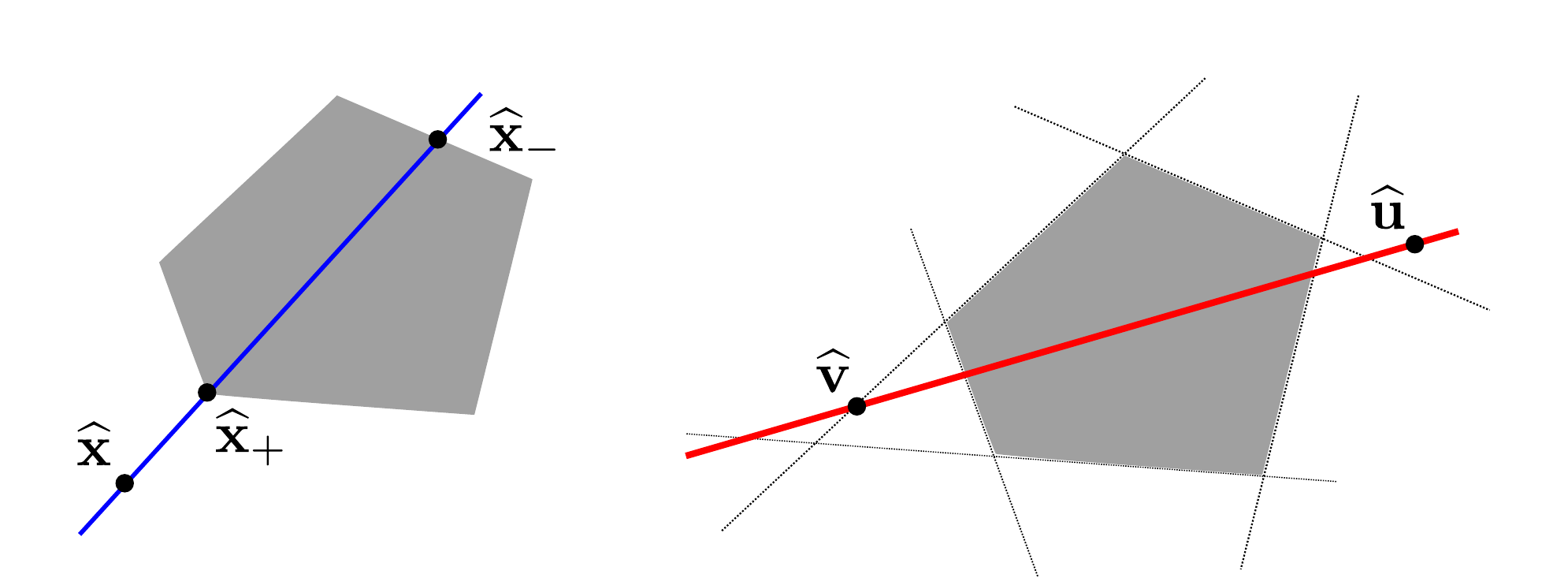}
		\caption{}\label{fig:xLvl1}
	\end{figure}
	
	Now assume that the representation space $V$ is not a Lorentz space.  Then,
	as noticed by Maxwell~\cite{maxwell1982}, there is a pair of
	\emph{orthogonal} vectors $\bu$ and $\bv$ such that $\cB(\bu,\bu)<0$ and
	$\cB(\bv,\bv)=0$.  For any linear combination $\bx=\lambda\bu+\mu\bv$ we have
	$\cB(\bx,\bx)<0$ as long as $\lambda\ne 0$.  The subspace $\Span\{\bu,\bv\}$
	intersect the hyperplanes $\Span\Delta_I$ at a ray $\RR_+ \bx_I$ for every
	$1$-facial subset $I \subset S$; see Figure~\ref{fig:xLvl1} (right) for the
	projective picture (modeled on $H^1_\bt$).  This means that each of these
	supporting hyperplanes contains a time-like ray $\RR_+ \bx_I$, with at most
	one exception (namely $\RR_+ \bv$).  This contradicts the fact that
	$(\Delta,\Phi)$ is of level $1$.  We then proved that $(\Delta,\Phi)$ is
	Lorentzian.

	Let $I$ be any $1$-facial subset of $S$. Since $(\Delta,\Phi)$ is of level~$1$,
	the subspace $\Span\Delta_I$ is not time-like, so its orthogonal companion
	$\omega_I\RR$ is not space-like.  This proves that no vector in $\Delta^*$ is
	space-like.

	It is clear that any two vectors in $\Delta^*$ span a Lorentz space.  For the
	disjointness, we only needs to prove that $\cB(\omega_I,\omega_J)\le 0$ for
	any two $1$-facial subsets $I\ne J\subset S$.  Since $\omega_I$ is not
	space-like, we have seen that $\omega_I \in \cC$, so $\cB(\omega_I,\omega_J)$
	has the same sign (possibly $0$) for all $1$-facial $J\subset S$, which is
	$\le 0$ when $\omega_I$ is time-like (take the sign of $\cB(\omega_I,
	\omega_I)$).  If $\omega_I$ is light-like, notice that $\omega_I \in
	\Cone\Delta_I$, i.e.\ $\omega_I$ can be written as a linear combination of
	$\Delta_I$ with coefficients of the same sign (possibly $0$).  For any
	$s\notin I$, we have $\cB(\omega_I, \alpha_s)>0$ and $\cB(\alpha_t,
	\alpha_s)\le 0$ for any $t \in I$, so $\omega_I$ must be a negative
	combination of $\Delta_I$.  We then conclude that $\cB(\omega_I,\omega_J) \le
	0$ since $\cB(\alpha_t,\omega_J)$ are all $\ge 0$.
\end{proof}

As a consequence, the Tits cone of a level-$1$ root system equals the set
of non-space-like vectors; see also~\cite[Proposition~9.4]{dyer2013}.

\begin{proposition}[extending \cite{maxwell1982}*{Proposition 1.6}]\label{prop:level2}
	If $(\Delta, \Phi)$ is of level~$2$, then it is Lorentzian, and	vectors in
	$\Delta^*$ are pairwise disjoint.  A vector $\omega_I \in \Delta^*$ is
	space-like if and only if the $1$-facial root system $(\Delta_I, \Phi_I)$ is
	of level $1$, in which case we have $\cB(\omega_I,\omega_I) \le 1$.
\end{proposition}

\begin{proof}
	Let $(\Delta,\Phi)$ be a canonical root system of level $2$ in $(V,\cB)$ with
	associated Coxeter system $(W,S)$.  If $V$ is of dimension $3$, it is
	immediate that $(\Delta,\Phi)$ is Lorentzian.  So we assume that the
	dimension of $V$ is at least $4$.

	We argue as in the proof of Proposition~\ref{prop:level1}.  If the
	representation space $V$ is not a Lorentz space, it contains a pair of
	orthogonal vectors $\bu$ and $\bv$ such that $\cB(\bu,\bu)<0$ and
	$\cB(\bv,\bv)=0$.  Then, for any linear combination $\bx=\lambda\bu+\mu\bv$,
	we have $\cB(\bx,\bx)<0$ as long as $\lambda\ne 0$.  The subspace
	$\Span(\bu,\bv)$ intersect the hyperplanes $\Span(\Delta_I)$ at a ray $\RR_+
	\bx_I$ for every $1$-facial subset $I \subset S$.  Since the $1$-facial root
	subsystems are all of level $\le 1$, we see from the proof of
	Proposition~\ref{prop:level1} that, whenever $\Span(\Delta_I)$ contains a
	time-like ray $\RR_+ \bx_I$, it must be in the interior of the facet
	$\Cone(\Delta_I)$.  In the projective picture (modeled on $H^1_\bt$),
	$\Span\{\bu,\bv\}$ appears as a line $L$ intersecting every facet of
	$\proj\cC$ such that every intersection point are in the interior of a facet,
	with at most one exception (namely $\proj \bv$) on the boundary of a facet.
	But $L$ intersects the boundary of $\proj \cC$ only at two points.  Now that
	one is in the interior of a facet, the other must be the intersection of all
	other facets, hence a vertex.  The only possibility is that $\proj \cC$ is a
	pyramid, in which case $L$ pass through the apex and an interior point of the
	base facet.  But the apex is a projective simple root $\proj \alpha$, and
	$\cB(\alpha,\alpha)>0$, so $\proj \alpha \notin L$.  This contradiction
	proves that $V$ is a Lorentz space.

	Let $\bx$ be a vector with $\cB(\bx,\bx) \le 0$ and assume $h(\bx)\ge 0$.
	Contrary to the case of level $1$, it is possible that $\proj\bx \notin
	\proj\cC$.  In this case, we can again write $\bx=\bx_+-\bx_-$ where $\bx_+
	\in \Cone(\Delta_I)$ and $\bx_- \in \Cone(\Delta_J)$ and $I$, $J$ are two
	disjoint facial subset of $S$.  And again, since $\cB(\bx_+,\bx_-)<0$, we
	have either $\cB(\bx_+,\bx_+)$ or $\cB(\bx_-,\bx_-)<0$.  Since
	$(\Delta,\Phi)$ is of level $2$, the $1$-facial root subsystems are all of
	level $\le 1$, so we must have either $\bx_+$ or $\bx_-$ in the interior of a
	facet of $\cC$ by Proposition~\ref{prop:level1}.  In the projective picture,
	it means that for \emph{any} line through $\proj\bx$ that intersects two
	disjoint faces of $\proj \cC$, one intersection point must be in the interior
	of a facet.  We then conclude from Lemma~\ref{lem:radon} that
	$\cB(\bx,\omega)$ are non zero, and have the same sign for all but one
	$\omega\in\Delta^*$.  On the other hand, if $\proj\bx \in \proj\cC$, only
	when $\cB(\bx,\bx)=0$ is it possible that $\proj\bx$ lies on a
	codimension-$2$ face of $\proj\cC$.

	Consequently, the intersection $\Span\Delta_I \cap \Span\Delta_J$ is not
	time-like for any two $1$-facial subsets $I \ne J \subset S$.  The orthogonal
	companion of the intersection is the subspace $\Span\{\omega_I, \omega_J\}$,
	which is not space-like.  For proving the disjointness, one still needs to
	prove that $\cB(\omega_I,\omega_J)\le 0$.

	Assume that $\omega_I$ is not space-like.  A similar argument as in the proof
	of Proposition~\ref{prop:level1} shows that $\cB(\omega_I, \omega_J)\le 0$
	for all $J\ne I$ with at most one exception.  Let $K\ne I$ be this exception.
	Pick a generator $s \in I \setminus K$, we can write $\omega_I =
	\lambda\alpha_s - \omega'_I$, where $\omega'_I$ is a linear combination of
	$\Delta_K$ with coefficients of same sign, which is also the sign of
	$\lambda$.  We have $\cB(\alpha_s, \omega_I) = 0$ by definition, but this is
	not the case since $\cB(\alpha_s, \alpha_t) \le 0$ for $t \in K$ while
	$\cB(\alpha_s, \alpha_s) = 1$.  Therefore, the exception $K$ does not exist.

	If $\omega_I$ is space-like, then the subspace $\Span\Delta_I$ is
	time-like, so $(\Span\Delta_I, \cB)$ is a (non-degenerate) Lorentz space.
	This proves that $(\Delta_I, \Phi_I)$ is of level $1$.  Then, for a simple root
	$\alpha \notin \Delta_I$, let 
	\[
		\alpha'=\alpha-\frac{\cB(\alpha,\omega_I)}{\cB(\omega_I,\omega_I)}\omega_I
	\]
	be the projection of $\alpha$ on $\Span\Delta_I$.  Since
	$\cB(\alpha',\beta) \le 0$ for all $\beta \in \Delta_I$, $\proj\alpha'$ is in
	the Coxeter polytope $\proj \cC^*_I$ for $(\Delta_I,\Phi_I)$.  Since
	$(\Delta_I, \Phi_I)$ is of level $1$, $\alpha'$ is not space-like by
	Proposition~\ref{prop:level1}, i.e.
	\[
		\cB(\alpha',\alpha')=\cB(\alpha,\alpha)-\frac{\cB(\alpha,\omega_I)^2}{\cB(\omega_I,\omega_I)}
		=1-\frac{\cB(\alpha,\omega_I)^2}{\cB(\omega_I,\omega_I)}\le 0,
	\]
	which proves that 
	\begin{equation}\label{eqn:omega<1}
		\cB(\omega_I,\omega_I) \le \min_{\alpha\notin\Delta_I} \cB(\alpha,\omega_I)^2 = 1.
	\end{equation}
	Let $J$ be a $1$-facial subset such that $\alpha
	\in \Delta_J$, and
	\[
		\omega'_J=\omega_J-\frac{\cB(\omega_J,\omega_I)}{\cB(\omega_I,\omega_I)}\omega_I
	\]
	be the projection of $\omega_J$ on $\Span\Delta_I$. Then, since $\alpha'
	\in \proj \cC^*$,
	\[
		\cB(\alpha',\omega'_J)=\cB(\alpha,\omega_J)-\frac{\cB(\alpha,\omega_I)\cB(\omega_J,\omega_I)}{\cB(\omega_I,\omega_I)}
		=-\frac{\cB(\alpha,\omega_I)\cB(\omega_J,\omega_I)}{\cB(\omega_I,\omega_I)}=\cB(\alpha',\omega_J)\ge 0,
	\]
	which proves that $\cB(\omega_I,\omega_J) \le 0$.  Since $J$ can be chosen as
	any $1$-facial subset $J\ne I\subset S$, this finish the proof of disjointness.
\end{proof}

The proposition has an interesting consequence.
\begin{corollary}\label{cor:lev2-1}
	Let $\Delta$ be a canonical root basis of level $2$ in $(V,\cB)$, then the
	set
	\[
		\Delta\cup\{-\omega/\sqrt{\cB(\omega,\omega)}\mid \omega \in \Delta^*,
	\cB(\omega, \omega)>0 \}
\]
is a canonical root basis of level $1$.
\end{corollary}

The following proposition is proved for free root basis in, for instance,
\citelist{\cite{bourbaki1968}*{Ch.~V, \S~4.4, Theorem~1}
\cite{abramenko2008}*{Lemma 2.58}}.  An extension for canonical root basis can
be found in \cite{krammer2009}*{Theorem 1.2.2(b)}, who refers to
\cite{vinberg1971} for proof.  Note that $\cC^*$ is closed in the present
paper, so the inequalities are not strict.
\begin{proposition}[extending \cite{maxwell1982}*{Corollary 1.8}]\label{prop:maxwell18}
	For $\bx \in \cC^*$, $w\in W$ and $\alpha_s \in \Delta$, either $\cB(w(\bx),
	\alpha_s) \ge 0$ and $\length(sw) > \length(w)$, or $\cB(w(\bx), \alpha_s)
	\le 0$ and $\length(sw) < \length(w)$.
\end{proposition}

\begin{theorem}[extending \cite{maxwell1982}*{Theorem 1.9}]\label{thm:level}
	The followings are equivalent:
	\begin{enumerate}
			\renewcommand{\labelenumi}{(\alph{enumi})}
			\renewcommand{\theenumi}{\labelenumi}
		\item $(\Delta,\Phi)$ is of level $1$ or $2$;\label{thm:levela}

		\item $(\Delta,\Phi)$ is Lorentzian and any two vectors in $\Omega$ are
			disjoint.\label{thm:levelb}
	\end{enumerate}
\end{theorem}

\begin{proof}
	Maxwell's proof applies with slight modification.  

	\ref{thm:levela}$\Rightarrow$\ref{thm:levelb}:  We only need to prove the
	disjointness.  We first prove that, for any $\omega_I, \omega_J \in \Delta^*$
	and $w \in W$,
	\begin{equation}\label{eqn:maxwell15}
		\cB(\omega_I,w(\omega_J))\le 0
	\end{equation}
	by induction on the length of $w\in W$.  The case of $w=e$ is already known.
	One may assume that $\length(tw)>\length(w)$ for all $t\in I$, otherwise one may
	replace $w$ by $tw$ in \eqref{eqn:maxwell15}.  So $w=sw'$ for some $s \notin
	I$ and $\length(w)>\length(w')$.  We then have
	\[
		\cB(\omega_I,w(\omega_J))=(s(\omega_I),w'(\omega_J))
		=\cB(\omega_I,w'(\omega_J))-2\cB(\alpha_s,\omega_I)\cB(\alpha_s,w'(\omega_J)).
	\]
	If $\omega_I \ne w'(\omega_J)$, \eqref{eqn:maxwell15} is proved since
	$\cB(\omega_I,w'(\omega_J))\le 0$ by inductive hypothesis,
	$\cB(\alpha_s,\omega_I)>0$ by definition, and $\cB(\alpha_s,w'(\omega_J)) \ge
	0$ by Proposition~\ref{prop:maxwell18}.  Otherwise, if $\omega_I =
	w'(\omega_J)$, we have 
	\[
		\cB(\omega_I,w(\omega_J)) =\cB(\omega_I,\omega_I)-2\cB(\alpha_s,\omega_I)^2 \le 0
	\]
	by \eqref{eqn:omega<1}.

	It remains to prove that $\cB$ is not positive definite on the subspace
	$\Span(\omega_I,w(\omega_J))$.  Assume the opposite.  Then $\omega_I$ is
	space-like hence $(\Delta_I,\Phi_I)$ is of level~$1$.  Let $\bv$ be the
	projection of $w(\omega_J)$ on $\omega_I^\perp=V_I$.  The subspace
	$\bv^\perp$ in $V_I$ is the time-like intersection $\omega_I^\perp \cap
	w(\omega_J)^\perp$ in $V$, so $\bv$ must be space-like.  On the other hand,
	for all $t\in I$, $\cB(\alpha_t,w(\omega_J)) \ge 0$ because $\length(tw) \ge
	\length(w)$.  We then conclude that $\bv$ is in the Coxeter polytope of
	$(\Delta_I, \Phi_I)$, so $\bv$ must be time-like.  This contradiction
	finishes the proof of disjointness. 

	\ref{thm:levelb}$\Rightarrow$\ref{thm:levela}: Since $\cB$ is not positive
	definite on $\Span\{\omega_I,\omega_J\}$ for any $\omega_I \ne \omega_J \in
	\Delta^*$, the orthogonal companion of these subspaces, which contains the
	codimension-$2$ faces of $\cC^*$, are not time-like.  So $\cB$ is positive
	semidefinite on all codimension-$2$ faces, which proves that $(\Delta,\Phi)$
	is of level $1$ or $2$.
\end{proof}

\subsection{Infinite ball packings}\label{ssec:packing}

For a \emph{space-like} vector $\bx$ in the Lorentz space $(V,\cB)$, the
\emph{normalized vector} $\normal{\bx}$ of $\bx$ is given by
\[ 
  \normal\bx=\bx/\sqrt{\cB(\bx,\bx)}.
\]
It lies on the one-sheet hyperboloid $\mathcal{H}=\{\bx\in V
\mid\cB(\bx,\bx)=1\}$.  Note that $\proj\bx=\proj{-\bx}$ is the same point in
$\mathbb{P}V$, but $\normal\bx$ and $\normal{-\bx}$ are two different vectors in
opposite directions in~$V$.  One verifies that two space-like vectors $\bx,\by$
are disjoint if and only if $\cB(\normal\bx,\normal\by) \le -1$. 

\begin{figure}[htb]
	\centering
	\includegraphics[width=.4\textwidth]{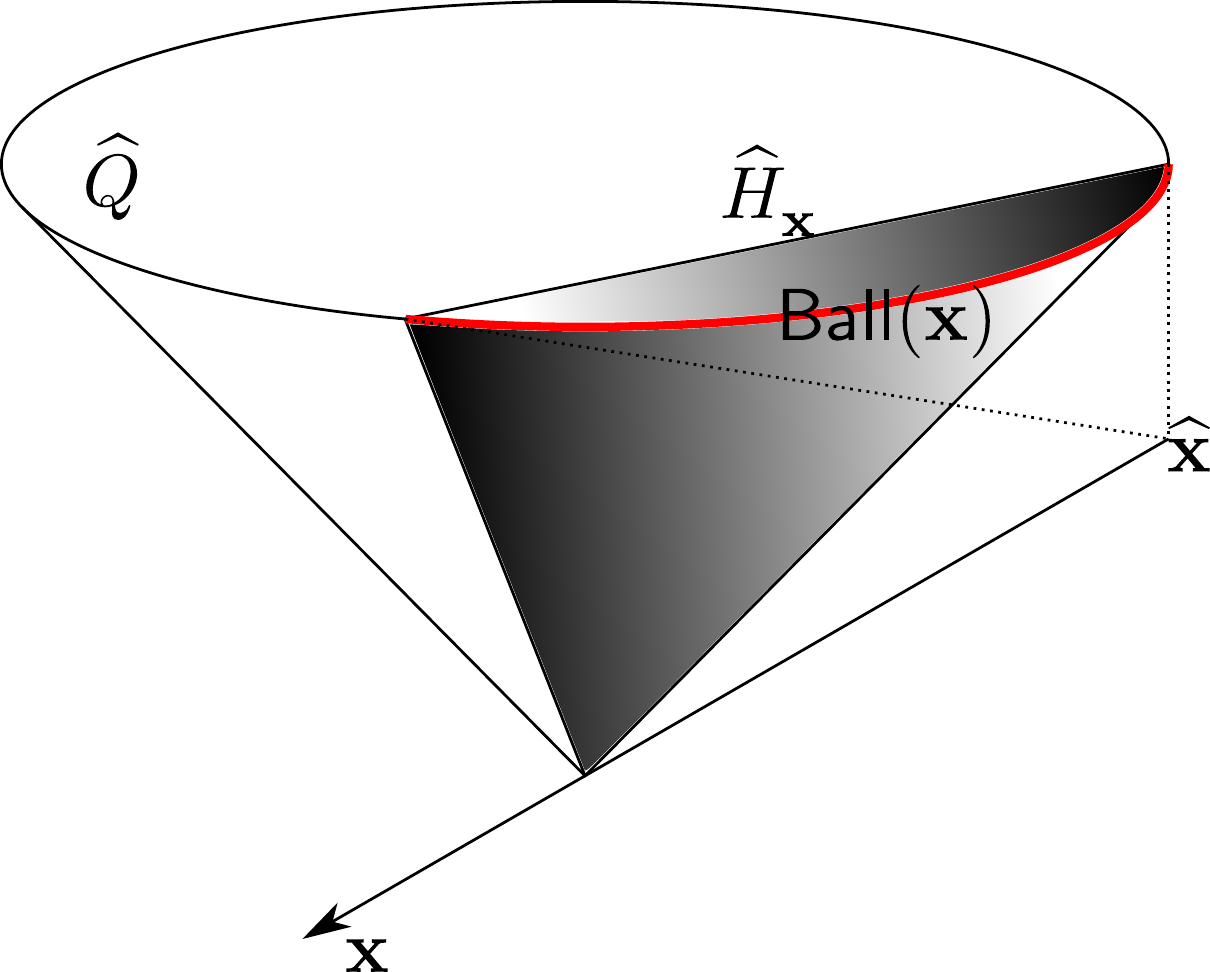}
	\caption{\label{fig:ball}}
\end{figure}

A correspondence between space-like directions in $(d+2)$-dimensional Lorentz
space $(V,\cB)$ and $d$-dimensional balls is introduced in
\cite{maxwell1982}*{\S 2}, see also
\citelist{\cite{hertrich-jeromin2003}*{\S~1.1} \cite{cecil2008}*{\S~2.2}}.  Fix
a time-like direction of past $\bt$ so that the projective light cone $\proj Q$
appears as a closed sphere on $H_\bt^1$.  Then in the affine picture, given a
space-like vector~$\bx\in V$, the intersection of $\proj Q$ with the half-space
$\proj H_\bx^-=\{\bx'\in H_\bt^1\mid\cB(\bx,\bx')\le 0\}$ is a closed ball
(spherical cap) on $\proj Q$. We denote this ball by $\ball(\bx)$; see
Figure~\ref{fig:ball}.  After a stereographic projection, $\ball(\bx)$ becomes
an $d$-dimensional ball in Euclidean space.  Here, we also regard closed
half-spaces as closed balls of curvature $0$, and complement of open balls as
closed balls of negative curvature.  For two past-directed space-like vectors
$\bx$ and $\by$, we have
\begin{itemize}
  \item $\ball(\bx)$ and $\ball(\by)$ are disjoint if $\cB(\normal \bx,\normal
  	\by)<-1$; 

  \item $\ball(\bx)$ is tangent to $\ball(\by)$ if $\cB(\normal \bx,\normal
  	\by)=-1$;

  \item $\ball(\bx)$ and $\ball(\by)$ \emph{overlap} (i.e. their interiors
  	intersect) if $\cB(\normal \bx,\normal \by)>-1$; 

  \item $\ball(\bx)$ and $\ball(\by)$ \emph{heavily overlap} (i.e. their
  	boundary intersect transversally at an obtuse angle, or one is contained in
  	the other) if $\cB(\normal\bx,\normal\by)>0$.

  \item One of $\ball(\bx)$ and $\ball(\by)$ is contained in the other if
  	$\cB(\normal\bx,\normal\by)\ge 1$.
\end{itemize}

A \emph{ball packing} is a collection of balls with disjoint interiors.  It is
then clear that a ball packing correspond to a set of space-like vectors $X\in
V$, with at most one future-directed vector, such that any two vectors are
disjoint.  Conversely, Maxwell proved that every such set of space-like vectors
correspond to a ball packing \cite{maxwell1982}*{Proposition 3.1}.  So
Theorem~\ref{thm:packing}, which we restate below, follows directly from
Theorem~\ref{thm:level}.

\begin{theorem}[extending \cite{maxwell1982}*{Theorem 3.2}]\label{thm:packing2}
	Let $\Omega_r$ be the set of space-like vectors in $\Omega$, then
	$\{\ball(\omega) \mid \omega \in \Omega_r\}$ is a ball packing if and only if
	the associated Lorentzian root system is of level~$2$.
\end{theorem}

\begin{remark}
	It is not obvious that the statement is equivalent to
	Theorem~\ref{thm:packing}.  A vector $\omega \in \Omega_r$ is in the form of
	$w(\omega_I)$ for some $w\in W$ and $\omega_I \in \Delta^*$.  The stabilizer
	of $\omega$ is then the parabolic facial subgroup in the form of $w W_I
	w^{-1}$.  Since $\omega_I$ is space-like, $W_I$ is of level $1$, and its
	Coxeter graph must be connected.  Facial subgroups with no finite irreducible
	components are said to be \emph{special} in~\cite{dyer2013}; see
	also~\cite{looijenga1980}.  Hence by~\cite{dyer2013}*{\S 10.3}, vectors in
	$\Omega_r$ span the space-like extreme rays of the Tits cone $\cT$, which
	proves the equivalence.  As we will see later, this does not hold for root
	systems of level $\ge 3$.
\end{remark}

In the following, the ball packing in the theorem is referred to as the
\emph{Boyd--Maxwell packing} associated to the Lorentzian root system.  A ball
packing is maximal if one can not add any additional ball into the packing
without overlapping other balls.

\begin{theorem}[extending \cite{maxwell1982}*{Theorem 3.3} and \cite{maxwell1989}*{Theorem 6.1}]
	The Boyd--Maxwell packing associated to a level-$2$ Lorentzian root system is
	maximal.
\end{theorem}

\begin{proof}[Sketch of proof]
	The proof of \cite{maxwell1982}*{Theorem 3.3} applies directly for free root
	basis.  Hence, the theorem is proved as long as the condition
	$\overline{\Cone\Omega_r}=\overline{\Cone\Omega}$ is verified.

	For Lorentzian free root basis of level $\ge 2$, the condition is confirmed
	to be true in~\cite{maxwell1989}*{Theorem 6.1}.  The proof applies here with
	slight modification, so we only give a sketch.  What we need to prove is that
	if $\omega_I \in \Omega$ is not space-like, then $\omega_I \in
	\overline{\Cone\Omega_r}$.
	
	If $\cB(\omega_I,\omega_I)<0$, $W_I$ is finite.  Choose any space-like
	$\omega \in \Delta^*$.  Then $\bv=\sum_{w\in W_I}w(\omega)$ is fixed by
	$W_I$, hence $\bv$ is a scalar multiple of $\omega_I$.

	If $\cB(\omega_I,\omega_I)=0$, then $(\Delta_I,\Phi_I)$ is affine and
	$\omega_I$ spans the radical of $(V_I,\cB)$.  Again, choose any space-like
	$\omega \in \Delta^*$.  Then, by a similar argument as in the proof
	of~\cite{maxwell1989}*{Proposition 5.15}, we have $\omega_I \in
	\overline{\Cone(W_I\omega)}$.
\end{proof}

We have extended most of the major results of \cite{maxwell1982}.  We now
continue to extend results from \cite{chen2014}.  

The \emph{residual set} of a collection of balls is the complement of the
interiors of the balls.  With the modified height function $h(\bx)$ in
Section~\ref{ssec:proj}, the proofs in \cite{chen2014} applies directly.  We
then obtain Theorem~\ref{thm:limroots} for level $2$, and the following theorem
for level $\ge 3$:

\begin{theorem}[{extending \cite{chen2014}*{Theorem 1.1; \S 3.4}}]
	For a Lorentzian root system of level~$\ge 3$, $\{\ball(\omega) \mid \omega
	\in \Omega_r\}$ is a maximal collection of balls with no heavily overlapping
	pairs, and the set of limit roots is again equal to the residual set.
\end{theorem}

\begin{remark}
	This time, the theorem can not be formulated in terms of extreme rays as in
	Theorem~\ref{thm:packing}.  It is possible that the stabilizer of some
	space-like $\omega_I \in \Delta^*$ has a finite component, hence not a
	special subgroup.  In this case, the rays spanned by $W(\omega_I)$ might not
	be extreme for $\cT$.
\end{remark}

The projective polytopes in the orbit $W\cdot\proj\cC^*$ are called
\emph{chambers}.  Analogous to the situation of free root basis ($\proj\cC^*$ is
a simplex) \cite{abramenko2008}, the chambers form a cell decomposition $\sC$
of the projective Tits cone $\proj\cT$ called \emph{Coxeter complex}.  It is a
pure polyhedral cell complex of dimension $d-1$ (dimension of $\PP V$).  The
set of vertices of $\sC$ is $\proj\Omega$.  The $1$-cells of $\sC$ are called
\emph{edges}, and $(d-2)$-cells are called \emph{panels}.

Since $\proj\cC^*$ is the fundamental domain for the action of $W$ on the
projective Tits cone $\proj T$, the orbit of two different fundamental weights
are disjoint.  So the vertices of the Coxeter complex admits a coloring by
$\Delta^*$, i.e.~a vertex $u$ is colored by $\omega\in\Delta^*$ if $u \in W
\cdot \proj\omega$.  Panels are orbits of the facets of $\Delta^*$, therefore
they can be colored by the simple roots, i.e.~a panel is colored by $\alpha \in
\Delta$ if it is the orbit of the facet of $\Delta^*$ corresponding to
$\alpha$.

The \emph{tangency graph} of a ball packing takes the balls as vertices and the
tangent pairs as edges.  We now try to describe the tangency graph in term of
the Coxeter complex.  Vertices with time- or light-like colors are called
\emph{imaginary vertices}; vertices with space-like colors are called
\emph{real vertices} because they correspond to balls in the packing, and are
therefore vertices in the tangency graph.  An edge of the Coxeter complex
connecting two real vertices of color $\omega$ and $\omega'$ is said to be
\emph{real} if $\cB(\normal\omega,\normal\omega')=-1$.  Real edges correspond
to tangent pairs in the packing,  and are therefore edges in the tangency
graph.  For a Lorentzian root system of level $2$, real vertices colored by
$\omega \in \Delta^*$ such that $\cB(\omega,\omega)=1$ are said to be
\emph{surreal}.  Two distinct surreal vertices of the same color~$\omega$ are
said to be \emph{adjacent} if they are vertices of two chambers sharing a panel
of color $\alpha$ such that $\cB(\omega,\alpha)=1$.  One verifies that pairs of
adjacent surreal vertices are also edges in the tangency graph.

With the definitions above (compare \cite{chen2014}*{\S 3.3}), the following
theorem follows by modifying the proof of \eqref{eqn:maxwell15} in the same way
as in the proof of \cite{chen2014}*{Theorem~3.7}.

\begin{theorem}[extending \cite{chen2014}*{Theorem~3.7}]
	The tangency graph of the ball packing associated to a Lorentzian root system
	of level $2$ takes the real vertices of the Coxeter complex as vertices; two
	vertices $u$ and $v$ are connected in the tangency graph if and only if one
	of the following is fulfilled:
  \begin{itemize}
    \item $uv$ is a real edge of the Coxeter complex, in which case $u$ and $v$
    	are of different colors,
    \item $u$ and $v$ are adjacent surreal vertices, in which case $u$ and $v$
    	are of the same color.
  \end{itemize}
\end{theorem}

\begin{corollary}[extending \cite{chen2014}*{Corollary 3.8}]
	The projective Tits cone $\proj\cT$ of a Lorentzian root system of level~$2$
	is an edge-tangent infinite polytope, i.e.~its edges are all tangent to the
	projective light cone.  Furthermore, the $1$-skeleton of $\proj\cT$ is the
	tangency graph of the ball packing associated to the root system.
\end{corollary}

\section{Partial classification of level-$2$ root systems}\label{sec:Tumarkin}
To provide examples of new infinite ball packings, we devote this section to a
partial enumeration of based root systems of level~$2$.

The enumeration will be based on analysis of the Coxeter polytope, and the
results will be given in the form of Coxeter graphs.  For convenience, we will
talk about the level for Coxeter polytopes and (by abuse of language) for
Coxeter graphs, which is actually the level of the associated root system.  For
Coxeter graphs, this leads to a confusion with Maxwell's definition, but should
not cause any problem within this section.

We will enumerate the Coxeter $d$-polytopes with $d+2$ facets.  The enumeration
is implemented by computer programs.  The current section only contains the
main ideas of the procedures.  Readers are encouraged to consult previous
enumerations such as~\cite{kaplinskaja1974, esselmann1996, tumarkin2004}, since
our algorithm is based on these works.

\subsection{Preparation}

Recall that a simple system $\Delta$ in $(V,\cB)$ can be represented by the
Coxeter graph~$G$ with Vinberg's convention.  Simple roots correspond to
vertices of $G$.  If two simple roots $\alpha,\beta\in\Delta$ are not
orthogonal, they are connected by an edge, which is solid with label $3\le
m<\infty$, if $\cB(\alpha,\beta) = -\cos(\pi/m)$; with label $\infty$ if
$\cB(\alpha,\beta)=-1$; or dashed with label $-c$ if $\cB(\alpha,\beta) =
-c<-1$.  The label $3$ on solid edges are often omitted.

If we consider the Coxeter polytope $\cP$, then vertices of the Coxeter graph
$G$ correspond to facets of $\cP$. A solid edge of $G$ with integer label means
that the intersection of two facets is time-like; a solid edge with label
$\infty$ means that the intersection is light-like; and a dashed edge means
that the intersection is space-like.

Let $G$ be a Coxeter graph, $G_1$ and $G_2$ be two subgraphs of $G$.  We use
$G_1+G_2$ to denote the subgraph induced by the vertices of $G_1$ and $G_2$,
use $G_1-G_2$ to denote the subgraph induced by the vertices of $G_1$ that are
not in $G_2$.  A subgraph with only one vertex is denoted by the vertex.

For a based root system $(\Delta,\Phi)$, the \emph{corank} of it's Coxeter
polytope~$\cP$ is defined as the nullity of the Gram matrix $B$ of $\Delta$.  A
Coxeter polytope of dimension $d$ and corank $k$ has $d+k+1$ facets.  In
particular, a Coxeter polytope of corank $0$ is a simplex.  In this case, the
level of the Coxeter graph coincide with Maxwell's definition.  For
convenience, a Coxeter polytope of level $l$ and corank $k$ is said to be a
$(l,k)$-polytope, or $(l^s,k)$-polytope if the level is strict.  All
these notions and notations are also used for the corresponding Coxeter graphs.

In the affine picture of the projective space $\PP V$, the projective light
cone $\proj Q$ appears as a closed surface that is projectively equivalent to a
sphere.  We can consider the interior of the sphere (time-like part) as the
Klein model of the hyperbolic space.  With this point of view, a Coxeter
polytope $\cP\in\PP V$ is a hyperbolic polytope, and is the fundamental domain
of the hyperbolic reflection group generated by the reflections in its
facets~\cite{vinberg1985}.  By Proposition~\ref{prop:level1}, Coxeter polytopes
of level $1$ correspond to finite-volume hyperbolic polytopes, or even compact
if the level is strict.

Vinberg~\cite{vinberg1984} proved that there is no strict level-$1$ Coxeter
polytopes of dimension $30$ or higher, and Prokhorov~\cite{prokhorov1986}
proved that there is no level-$1$ Coxeter polytopes in hyperbolic spaces of
dimension $996$ or higher.  On the other hand, Allcock~\cite{allcock2006}
proved that there are infinitely many level-$1$ (resp.\ strictly level-$1$)
Coxeter polytopes in every hyperbolic space of dimension $19$ (resp.\ $6$) or
lower, which suggests that a complete enumeration of level-$1$ Coxeter
polytopes is hopeless.  

Nevertheless, there are many interesting partial enumerations.  The
$(1,0)$-polytopes have been completely enumerated by Chein~\cite{chein1969}.
They are hyperbolic simplices of finite volume.  The list of Chein also
comprises $(1^s,0)$-polytopes, which was first enumerated by
Lann\'er~\cite{lanner1950}.  The $(1,1)$-polytopes have been enumerated by
Kaplinskaja~\cite{kaplinskaja1974} for simplicial prisms,
Esselmann~\cite{esselmann1996} for compact polytopes and
Tumarkin~\cite{tumarkin2004} for finite-volume polytopes.  Tumarkin also
studied $(1^s,2)$- and $(1^s,3)$-polytopes~\cite{tumarkin2007, tumarkin2008}.
Mcleod~\cite{mcleod2013} finished the classification of all pyramids of
level~$1$.  Our algorithm of enumeration is based on these works.

In this section, we study Coxeter polytopes of level $2$.  Recall that a
Coxeter polytope is of level $2$ if all its edges are time-like or light-like,
but some vertices are space-like.

In view of Corollary~\ref{cor:lev2-1}, we deduce immediately from the result
of~\cite{prokhorov1986} that there is no level-$2$ Coxeter polytopes in
hyperbolic spaces of dimension $996$ or higher.  However, in the shadow
of~\cite{allcock2006}, there might be infinitely many level-$2$ Coxeter
polytopes in lower dimensions, so a complete classification may be hopeless.  A
$(2,0)$-graph is either a connected graph, or a disjoint union of an isolated
vertex and a $(1,0)$-graph.  The enumeration of connected $(2,0)$-graphs was
initiated in \cite{maxwell1982} and completed in \cite{chen2014}.  In this
section, we would like to enumerate $(2,1)$-graphs.  

A $k$-face of a $d$-polytope is said to be \emph{simple} if it is the
intersection of $d-k$ facets, or \emph{almost simple} if every face containing
it is simple.  A polytope is said to be $k$-simple (resp.\ almost $k$-simple)
if all its $k$-faces are simple (resp.\ almost simple).  For a Coxeter
polytope, the stabilizer of a time-like (resp.\ light-like) face is of finite
(resp.\ affine) type.  Therefore, every time-like (resp.\ light-like) face is
simple (resp.\ almost simple).  We then conclude the following proposition from
the definition of level.

\begin{proposition}\label{prop:simple}
	A Coxeter polytope of level $l$ (resp.\ strictly of level $l$) is
	almost $l$-simple (resp.\ $l$-simple).
\end{proposition}

From the Gale diagram~\citelist{\cite{grunbaum2003}*{\S 6.3}
\cite{tumarkin2004}*{\S 2}} and by Proposition~\ref{prop:simple}, we know that
there are three possibilities for the combinatorial type of a $(2,1)$-polytope:

\begin{itemize}
	\item a product of two simplices, abbreviated as $\triangle\times\triangle$;
	\item the pyramid over a product of two simplices, abbreviated as
	$\Pyramid(\triangle\times\triangle)$; \item the $2$-fold pyramid over a
		product of two simplices, abbreviated
as~$\Pyramid^2(\triangle\times\triangle)$.
\end{itemize}
We now analyse the three types separately.

\subsection{$\cP$ has the type of $\triangle\times\triangle$}

In this case, vertices of $\cP$ are all simple.  The Coxeter graph $G$ consists
of two parts, say $G_1$ and $G_2$, corresponding to the two simplices.  A
$k$-face of $\cP$ correspond to a subgraph of $G$ obtained by deleting $k+2$
vertices, including at least one vertex from both $G_1$ and $G_2$.

For a $(1,0)$-graph, a vertex is said to be \emph{ideal} if its removal leaves
an affine Coxeter graph.  For a $(2,0)$-graph, a vertex is said to be
\emph{real} if its removal leaves a $(1,0)$-graph.  If the $(2,0)$-graph is not
connected, then the isolated vertex is the only real vertex.

\begin{lemma}\label{lem:simpsimp}
	If a $(2,1)$-polytope has the type of $\triangle\times\triangle$, then
	\begin{enumerate}[label=\emph{(\roman*)}]
		\item Its Coxeter graph $G$ consists of two $(1,0)$-graphs $G_1$ and $G_2$
		    and they are connected to each other.
			
	  \item For any $v_1\in G_1$ and $v_2\in G_2$, the graphs $G_1+v_2$ and
	      $G_2+v_1$ are $(2,0)$-graphs.

	  \item If one of the simplices, say the one represented by $G_2$, is of
	  	dimension $>1$, then the level of $G_1$ and $G_1+v_2$ are strict, and
	  	$v_2$ is the only real vertex of $G_1+v_2$.
	\end{enumerate}
\end{lemma}

\begin{proof}
	The two simplices are represented by Coxeter graphs $G_1$ and $G_2$
	respectively.  By~\cite{vinberg1985}*{Theorem~3.1}, a subgraph of $G$
	correspond to a time-like face of $\cP$ if and only if it is of finite type.
	By~\cite{vinberg1985}*{Theorem~3.2}, a subgraph of~$G$ correspond to a
	light-like face of $\cP$ if and only if it is of affine type.  Since $G_1$
	and $G_2$ are of corank $0$ and do not correspond to any face of $\cP$, they
	are not of level $0$.  
	
	For any vertex $v_1$ of $G_1$, the graph $G_1-v_1$ is obtained from $G$ by
	removing at least $3$ vertices.  Therefore, it corresponds to a simple face
	of $\cP$ of dimension $\ge 1$.  We then conclude that $G_1$ is a
	$(1,0)$-graph (therefore connected).  Same argument applies to $G_2$.  $G =
	G_1 + G_2$ must be connected, otherwise the bilinear form $B$ has two
	negative eigenvalues, and the root system is not Lorentzian.

	For any $v_2\in G_2$, the graph $G_1+v_2$ is not of level $1$ because $G_1$
	is.  It is of corank~$0$ since its positive polytope is a simplex.  It is of
	level $2$ because further deletion of any two vertices leaves a level-$0$
	Coxeter graph corresponding to a simple face of $\cP$ of dimension $\ge 1$.
	Same argument applies to $G_2+v_1$.  

	If the simplex represented by $G_2$ is of dimension $>1$, then $G_2$ has more
	than two vertices. In this case, the dimensions of the faces mentioned above
	are all strictly $>1$, so the levels of $G_1$ and $G_1+v_2$ are strict.
	Furthermore, for any vertex $v_1\in G_1$, the graph $G_1+v_2-v_1$ is of level
	$0$ since it corresponds to a face of $\cP$ of dimension $\ge 1$.  So $v_2$
	is the only real vertex of $G_1+v_2$.
\end{proof}

\begin{remark}
	The same type of argument applies for many other lemmas in this section, and
	we will not repeat them in detail.  
\end{remark}

We now sketch the procedure for enumerating Coxeter polytopes of this type.  We
need to distinguish two sub-cases.

\subsubsection{One of the simplices is of dimension $1$}\label{ssec:prism} 

In this case, $\cP$ has the combinatorial type of a simplicial prism. In the
Coxeter graph, we may assume that the two vertices of $G_2$ correspond to the
base facets of $\cP$, while the vertices of $G_1$ correspond to lateral facets.
By Lemma~\ref{lem:simpsimp}(i), $G_2$ is a $(1,0)$-graph, so the vertices of
$G_2$ are connected by a dashed edge, meaning that the two base facets do not
intersect inside the hyperbolic space.  

A prism is \emph{orthogonally based} if one of the base facets is orthogonal to
all the lateral facets.  Any prism of level-$2$ can be cut into two
orthogonally based prisms, and two orthogonally based prisms can be spliced
into one prism if they share the same base.  Therefore, we only need to
consider orthogonally based prisms, as also argued by
Kaplinskaja~\cite{kaplinskaja1974}.  In the Coxeter graph of an orthogonally
based prism, one vertex of $G_2$ is not connected to any vertex of $G_1$.  By
Lemma~\ref{lem:simpsimp}(ii), deletion of this vertex leaves a $(2,0)$-graph.
Therefore, we construct a candidate graph for $(2,1)$-simplicial prism by
attaching a vertex to a real vertex of a $(2,0)$-graph with a dashed edge.

Any candidate graph obtained in this way represents a Coxeter polytope of level
$1$ or $2$.  In fact, by attaching a vertex $u$ to a real vertex $v$ with a
dashed edge, we are truncating the Coxeter simplex at a space-like vertex.  The
truncating facet intersects all the lateral faces orthogonally.  In the Coxeter
graph $G$, $u$ is the vertex for the truncating facet, $v$ is for the other
base facet.  Each vertex of the truncating facet corresponds to a graph of the
form $G_1+u-w$ where $w\in G_1$ and $u$ is isolated.  Such a graph is of level
$0$ because $G_1$ is of level $1$.  Consequently, if the truncated vertex is
the only space-like vertex, the Coxeter polytope we obtain has no space-like
vertex, and its level is $1$.  Otherwise, the polytope is of level $2$.

For each candidate graph, we still need calculate the label for the dashed
edge.  For this, we make use of the fact that the determinant of the matrix~$B$
is $0$.  The Coxeter polytope has the combinatorial type of a orthogonally
based simplicial prism if this label is $<-1$.  Otherwise, if the label $=-1$,
the dashed edge should be replaced by a solid edge with label $\infty$.  In
this case, the truncating facet truncates ``too much'' and meets another
vertex, so the Coxeter polytope has the combinatorial type of a pyramid over a
simplicial prism, which will be classified later.

The list of $(2,0)$-graphs with at least five vertices can be found in
\cite{chen2014}, where real vertices are colored in white or grey.  By attaching a
vertex to each of these real vertices, we obtain $655$ candidate Coxeter
graphs.  Among them, $129$ graphs correspond to pyramids over simplicial
prisms; $17$ graphs correspond to simplicial prisms of level $1$, as also
enumerated in~\cite{kaplinskaja1974}; the remaining $509$ graphs correspond to
orthogonally based prisms of level $2$.  Due to the large number of graphs, we
do not give the list in this paper. 

A Coxeter graph with three vertices is always of level $\le 2$, and the number
of real vertices equals the number of dashed edges.  
We then obtain level-$2$ Coxeter graphs in the form of
\tikz[scale=.8, baseline=-.5em]{
  \fill (0,0) circle (.1);
  \fill (150:1) circle (.1);
  \fill (210:1) circle (.1);
  \fill (1,0) circle (.1);
  \draw[dashed] (0,0)--(1,0);
  \draw[dashed] (150:1)--(210:1);
  \draw[dashed] (0,0)--(150:1);
  \draw (0,0)--(210:1);
}.
It corresponds to a two dimensional square.  A Coxeter graph with four vertices
is of level $\le 2$ as long as there is no dashed edge, and the number of real
vertices equals the number triangles representing hyperbolic triangle
subgroups.  This completes the classification of orthogonally based simplicial
prisms of level $1$ or $2$.

The complete list of prisms of level $1$ or $2$ is obtained by splicing two
orthogonally based prisms of level $1$ or $2$ if they share a same orthogonal
base.  In other words, if two Coxeter graphs of orthogonally based prism of
level $1$ or $2$ share the same subgraph $G_1$, we can identify this subgraph,
and merge the dashed edge into one, as shown below.  
\[
	\tikz[baseline=(G1.base)]{
		\node[draw,circle,inner sep=0] at (0,0) (G1){$G_1$};
		\node[fill,circle,inner sep=2] at (30:1) (v){};
		\node[fill,circle,inner sep=2] at (-30:1) (u){};
		\draw (G1) -- (v); \draw[dashed] (v)--(u);
	}
	+
	\tikz[baseline=(G1.base)]{
		\node[draw,circle,inner sep=0] at (0,0) (G1){$G_1$};
		\node[fill,circle,inner sep=2] at (30:1) (v){};
		\node[fill,circle,inner sep=2] at (-30:1) (u){};
		\draw (G1) -- (u); \draw[dashed] (u)--(v);
	}
	=
	\tikz[baseline=(G1.base)]{
		\node[draw,circle,inner sep=0] at (0,0) (G1){$G_1$};
		\node[fill,circle,inner sep=2] at (30:1) (v){};
		\node[fill,circle,inner sep=2] at (-30:1) (u){};
		\draw (v)--(G1)--(u); \draw[dashed] (u)--(v);
	}
\]
The result is of level $1$ if the two orthogonally based prisms are both
$(1,1)$-polytopes, or of level $2$ otherwise.

\subsubsection{The two simplices are both of dimension $>1$}

A vertex $u$ of a $(1,0)$-graph $H$ is called a \emph{port} of $H$ if there is
a $(2,0)$-graph in the form of $H+v$ in which $v$ is the only real vertex and
$u$ is a neighbor of $v$.

We construct a candidate $(2,1)$-graph by connecting the ports of two
$(1^s,0)$-graphs in all possible ways that satisfy Lemma~\ref{lem:simpsimp}.
For each candidate, we calculate its corank, and verify its level of $G$ by
checking the level of $G-v_1-v_2$ for each $v_1\in G_1$ and $v_2\in G_2$.
Recall that $G$ is of level $2$ if the level of $G-v_1-v_2$ is always $\le 1$
but not always $0$.

In practice, ports are detected with the help of the following lemma.
\begin{lemma}[Extending \cite{esselmann1996}*{Lemma 4.2}]\label{lem:trick}
	If $u$ is a port of $H$, then there is a $(2,0)$-graph in the form of $H+v$
	in which $v$ is the only real vertex and is only connected to $u$ by a solid
	edge with label $3$.
\end{lemma}
For each port $u$ of $H$, we find all the $(2^s,0)$-graphs in the form of $H+v$
in which $v$ is the only real vertex and $u$ is a neighbor of $v$.  These
$(2^s,0)$ graphs indicate the possible ways for connecting $H$ to other
$(1,0)$-graphs.

Some $(1^s,0)$-graphs with ports are listed in Figure~\ref{fig:L1Ports}, where
ports are colored in white and marked with numbers.  We exclude hyperbolic
triangle groups with label $7$, $9$ or $\ge 11$.  By the same technique as
in~\cite{esselmann1996}*{\S~4.1, Step 3) 4)}, we verified by computer that
these triangle groups can not be used to form any Coxeter graph of positive
corank.  In Table~\ref{tbl:Pyr0-nn}, we list all the $28$ Coxeter polytopes of
the type $\triangle\times\triangle$, with both simplices of dimension $>1$.
For each polytope, we give the position of $G_1$ and $G_2$ in
Figure~\ref{fig:L1Ports}, and the edges connecting $G_1$ and $G_2$ in the
format of (port in $G_1$, port in $G_2$, label).

\subsection{$\cP$ has the type of $\Pyramid(\triangle\times\triangle)$}

In this case, the Coxeter graph $G$ consists of three parts: a vertex
corresponding to the base facet, and two subgraphs $G_1$ and $G_2$
corresponding to the two simplices.  The base facet has the type of
$\triangle\times\triangle$.  Vertices on the base facet are all simple.  Every
$k$-face of $\cP$, except for the apex vertex, corresponds to a subgraph of $G$
obtained by deleting $k+2$ vertices, including at least one vertex from both
$G_1$ and $G_2$.  The stabilizer of the apex is represented by the graph
$G_1+G_2$.  The corank of $G_1+G_2$ is $1$, and its level may be $0$ if the
apex is light-like, or $1$ if the apex is space-like.  We now study these two
sub-cases separately.

\subsubsection{The apex is light-like}

In this case, $G_1+G_2$ is an affine graph.

If a $(1,0)$-graph has a unique ideal vertex, we call this vertex the
\emph{hinge} of the graph.
\begin{lemma}\label{lem:pyrsimpsimp}
	If $G_1+G_2$ is a $(0,1)$-graph, then
	\begin{enumerate}[label=\emph{(\roman*)}]
		\item $G_1$ and $G_2$ are both affine and are not connected to each
			other;

		\item $G_1+v$ and $G_2+v$ are both $(1,0)$-graphs;

		\item If one of the simplices, say the one represented by $G_2$, is of
			dimension $>1$, then $v$ is the hinge of $G_1+v$;

		\item Let $v_2\in G_2$ be a neighbor of $v$, then $G_1+v+v_2$ is a
			$(2,0)$-graph.
	\end{enumerate}
\end{lemma}

For a proof, the first point follows from the same argument as in the proof
of~\cite{tumarkin2004}*{Lemma~4}, and other points follows from the same type
of argument as in the proof of Lemma~\ref{lem:simpsimp}.  We now sketch the
procedure for enumerating Coxeter polytopes of this type.

If one of the simplices is of dimension $1$, we construct a candidate
$(2,1)$-graph as follows.  For an ideal vertex $v$ of a non-strict
$(1,0)$-graph $H$, we extend $H$ to a $(2,0)$-graph by attaching a vertex $u$
to $v$ with a solid edge of label $a$.  We allow $a$ to be $2$, meaning that
$u$ and $v$ are actually not connected.  We attach a second vertex $u'$ to $v$
in a second (possibly the same) way with label $a'$.  Then we connect $u$ and
$u'$ by a solid edge with label $\infty$.  In the graph we obtain, $v$
correspond to the base facet of $\cP$, $H-v$ correspond to $G_1$ and $u+u'$
correspond to $G_2$.  Since $v+u+u'$ is a $(1,0)$-graph, $a$ and $a'$ can not
be both $2$.  One then verifies Lemma~\ref{lem:pyrsimpsimp} on $H+u+u'$, and
conversely that any graph obtained in this way is of level $1$ or $2$.
Furthermore, $H+u+u'$ has a positive
corank~\cite{tumarkin2004}*{Lemma~3}\footnotemark[1] which necessarily equals
$1$ because $H+u$ is of corank~$0$.  With the same argument as in
Section~\ref{ssec:prism}, we see that the graph is of level $2$ as long as $u$
and $u'$ are not both the only real vertex of $H+u$ and $H+u'$ respectively.
\footnotetext{
	This lemma follows from Proposition 12 of \cite{vinberg1984}, instead
	of~\cite{vinberg1967}.
}

The list of non-strict $(1,0)$-graph with $\ge 4$ vertices can be found
in~\cite{chein1969}.  The procedure above then gives, up to graph isomorphism,
$358$ graphs of level $1$ or $2$.  Among them, $89$ are of level $1$ as also
enumerated in~\cite{tumarkin2004}.  The remaining $269$ graphs correspond to
pyramids of level $2$, and $129$ of them were discovered earlier in
Section~\ref{ssec:prism} when enumerating orthogonally based simplicial prisms.
Due to the large number of graphs, we do not give the list in this paper.  

If both simplices are of dimension $1$, then the Coxeter graph is in the form
of
\tikz[scale=.8, baseline=-.5em]{
  \fill (0,0) circle (.1);
  \fill (150:1) circle (.1);
  \fill (210:1) circle (.1);
  \fill (30:1) circle (.1);
  \fill (-30:1) circle (.1);
	\draw (0,0) -- node[midway] {$a_1$} 
				(150:1)--node[midway] {$\infty$}
				(210:1)--node[midway] {$a_2$}
				(0,0) -- node[midway] {$b_1$} 
				(30:1) --node[midway] {$\infty$}
				(-30:1)--node[midway] {$b_2$} (0,0);
}.
It corresponds to a square pyramid.  Its level is at most $2$, and equals $2$
if $1/a_i+1/b_j<1/2$ for some $i,j\in\{1,2\}$.

If both simplices are of dimension $>1$, we construct a candidate $(2,1)$-graph
by taking two non-strict $(1,0)$-graphs with hinges and identifying their
hinges.  For any graph $G$ constructed in this way, one easily verifies
Lemma~\ref{lem:pyrsimpsimp}.  The corank of~$G$ is positive
by~\cite{tumarkin2004}*{Lemma~3}\footnotemark[1], and necessarily equals $1$
by applying~\cite{vinberg1984}*{Proposition~12} on $G-v$ for any $v$ different
from the hinge.  Finally, we verify the level of $G$ by checking the level of
$G-v_1-v_2$ for each $v_1\in G_1$ and $v_2\in G_2$.  Recall that $G$ is of
level $2$ if the level of $G-v_1-v_2$ is always $\le 1$ but not always $0$.

All non-strict $(1,0)$-graphs with a hinge and $\ge 4$ vertices are listed in
Figure~\ref{fig:L1Hinges}.  In Table~\ref{tbl:Pyr1-0-nn}, we list all the $65$
polytopes of this class by giving the position of $G_1+v$ and $G_2+v$ in
Figure~\ref{fig:L1Hinges} respectively.

\subsubsection{The apex is space-like}
In this case, $G_1+G_2$ is a $(1,1)$-graph.

For a $(3,0)$-graph $G$, a vertex is said to be \emph{surreal} if its removal
leaves a $(2,0)$-graph.
\begin{lemma}\label{lem:pyrsimpsimp1}
	If $G_1+G_2$ is a $(1,1)$-graph, then
	\begin{enumerate}[label=\emph{(\roman*)}]
		\item $G_1$ and $G_2$ are both $(1^s,0)$-graphs, and they are connected;

		\item $G_1+v$ and $G_2+v$ are $(2^s,0)$-graphs in which $v$ is the only
			real vertex;
	\end{enumerate}
	and for any $v_2\in G_2$,
	\begin{enumerate}[resume, label=\emph{(\roman*)}]
		\item the graph $G_1+v_2$ is a $(2^s,0)$-graph in which $v_2$ is the only
			real vertex;

		\item the graph $G_1+v+v_2$ is a $(3^s,0)$-graph for which $v$ and $v_2$
		    are the only surreal vertices.
	\end{enumerate}
\end{lemma}
For a proof, the first point is~\cite{tumarkin2004}*{Lemma~2(I)}, and other
points follows from the same type of argument as in the proof of
Lemma~\ref{lem:simpsimp}.  We now sketch the procedure for enumerating Coxeter
polytopes of this type.

If one of the simplices is of dimension $1$, $G_1+G_2$ represents a
$(1,1)$-prism.  Coxeter graphs for $(1,1)$-prisms are classified
in~\cite{kaplinskaja1974}, where a list of orthogonally based $(1,1)$-prisms is
given.  Each graph in the list is obtained by attaching a dashed edge to the
unique real vertex of a $(2^s,0)$-graph, see also~\cite{vinberg1985}*{\S~5.4}.
Therefore, if we ignore the dashed edges, the list in~\cite{kaplinskaja1974}
essentially classified all connected $(2^s,0)$-graphs with a unique real
vertex.

Given a $(1^s,0)$-graph $H$, we construct a candidate $(2,1)$-graph as follows.
We extend $H$ to three $(2^s,0)$-graphs (possibly same) $H+u$, $H+v$ and $H+w$,
in which $u$, $v$ and $w$ are respectively the unique real vertex, therefore
$H+u+v+w$ satisfies Lemma~\ref{lem:pyrsimpsimp1}(iii).  We now choose $u$ as
the base facet.  We combine the three graphs by identifying $H$, then connect
$v$ and $w$ by a dashed edge, such that $G_1+v+w$ represent a $(1,1)$ prism.
We also connect $u,v$ and $u,w$ respectively by a solid edge, and give it all
possible labels (necessarily between $2$ and $6$) such that $H+u+v$ and $H+u+w$
are $(3^s,0)$-graphs.  One verifies that $u,v$ and $u,w$ are respectively the
only two surreal vertices, so $H+u+v+w$ satisfies
Lemma~\ref{lem:pyrsimpsimp1}(iv).  Finally, we verify the corank and the level
of the candidate graph.

In order to satisfy Lemma~\ref{lem:pyrsimpsimp1}(i), $v$ and $w$ can not be
both disjoint from $H$.  But the vertex $u$ can be an isolated vertex, in which
case $H+u+v+w$ is indeed a $(2,1)$-pyramid.  Otherwise, we list in
Figure~\ref{fig:Pyr1-1-1n} all the $18$ connected $(2,1)$-pyramid in dimension
$\ge 5$.
\begin{figure}[h]
	\includegraphics[width=\textwidth]{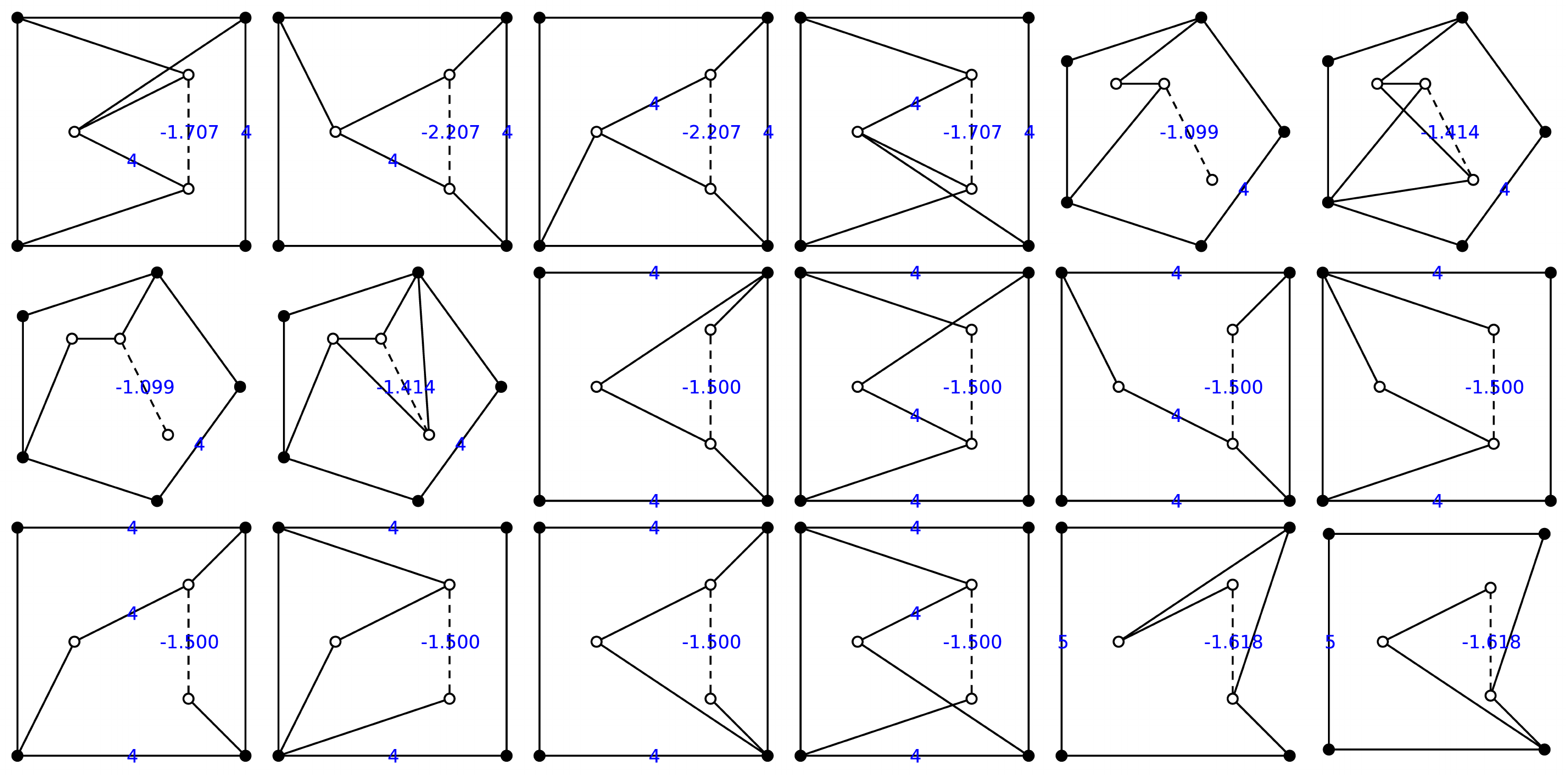}
	\caption{
		The $18$ connected $(2,1)$-graphs of rank $\ge 7$ whose Coxeter polytope
		has the type of a pyramid over a prism with space-like apex.
	} \label{fig:Pyr1-1-1n}
\end{figure}
For $4$-dimensional pyramids over triangular prisms, we obtain $266$ connected
$(2,1)$-graphs from triangle graphs with labels at most~$6$.  Because of the
large number of graphs, we do not list them in this paper.  For triangle graphs
with a label $k\ge 7$, the Coxeter graph is necessarily in the form of
\tikz[scale=.8, baseline]{
  \fill (-1,-.5) circle (.1);
  \fill (-1,.5) circle (.1);
  \fill (0,0) circle (.1);
  \fill (0.5,0) circle (.1);
  \fill (1,-.5) circle (.1);
  \fill (1,.5) circle (.1);
	\draw (0.5,0)--(0,0)--(1,-.5)--(0.5,0)--(1,.5)--(0,0);
	\draw (-1,.5)--(0,0)--(-1,-.5);
	\draw (-1,.5) -- node[midway] {$k$} (-1,-.5); 
	\draw[dashed] (1,.5) -- node[midway] {$-c$} (1,-.5);
}.  
The unlabeled edges can not have label $\ge 7$, so for a given $k$, there are
only finitely many possibilities for the labels.  For each of them, the value
of $-c$ is determined by $k$.  We then use Sage to find the expressions of the
determinant in terms of $k$, and find no integer root that is $\ge 7$ for these
expressions.  So we believe that the labels on solid edges are at most $6$ for
this type of $(2,1)$-graphs.  However, the author thinks that this is the point
to question the reliability of computer enumeration, and an analytic
explanation is welcomed.

For $3$-dimensional pyramids over squares, both simplices are of dimension $1$,
and the Coxeter graph is necessarily in the form of
\tikz[scale=.8, baseline]{
  \fill (0,1) circle (.1);
  \fill (150:1) circle (.1);
  \fill (210:1) circle (.1);
  \fill (30:1) circle (.1);
  \fill (-30:1) circle (.1);
	\draw (150:1)--(-30:1)--(210:1)--(30:1)--(150:1); 
	\draw[dashed] (150:1)--(210:1);
	\draw[dashed] (-30:1)--(30:1);
	\draw (150:1)--(0,1)--(30:1);
	\draw (210:1)--(0,1)--(-30:1);
}.
To be a $(2,1)$-graph, the dashed edges need to bear correct labels and the
corank should be $1$.  We do not have a complete characterisation for this
case.

If both simplices are of dimension $>1$, $G_1+G_2$ falls in the list
in~\cite{esselmann1996} and~\cite{tumarkin2004}.  The list contains eight
graphs. Each graph $G$ in the list is obtained by connecting two
$(1^s,0)$-graphs $G_1$ and $G_2$.  We extend $G_1$ and $G_2$ to two $(2^s,0)$
graphs $G_1+u_1$ and $G_2+u_2$ in which $u_1$ and $u_2$ are respectively the
unique real vertices (possibly isolated).  We then obtain a candidate
$(2,1)$-graph $G$ by identifying $u_1$ and $u_2$ as a single vertex $u$.
Finally, we calculate the corank, and verify the level of the candidate by
checking the level of $G-v_1-v_2$ for $v_1\in G_1$ and $v_2\in G_2$.  This
time, we only need the level of $G-v_1-v_2$ to be always $\le 1$.  If $u$ is an
isolated vertex, the result is clearly a $(2,1)$-graph.  Otherwise, there are
three connected $(2,1)$-graph.  They are listed in Figure~\ref{fig:Pyr1-1-nn},
where the white vertex correspond to the base facet.
\begin{figure}[h]
	\centering
	\includegraphics[width=.8\textwidth]{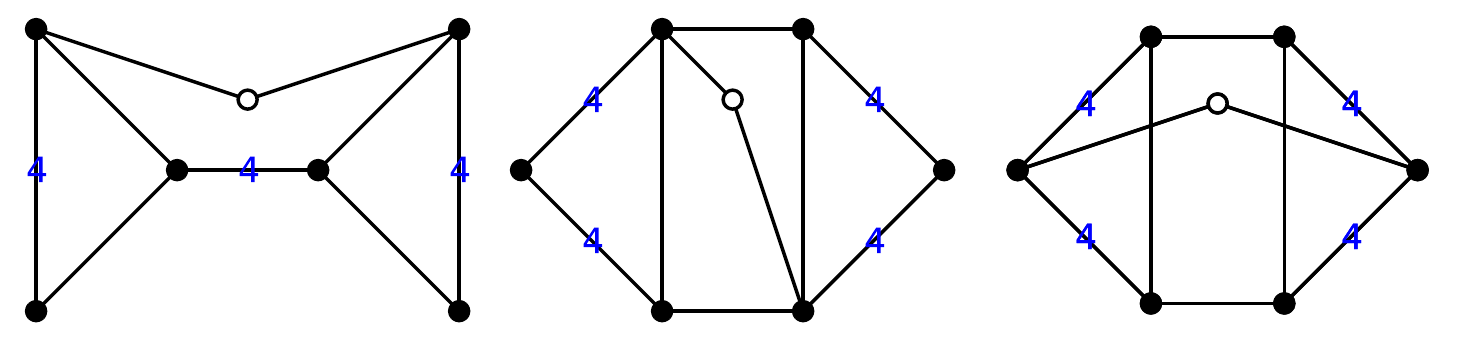}
	\caption{
	  The three connected $(2,1)$-graphs whose Coxeter polytope has the type of
	  $\Pyramid(\triangle\times\triangle)$ with space-like apex.
	} \label{fig:Pyr1-1-nn}
\end{figure}

\subsection{$\cP$ has the type of $\Pyramid^2(\triangle\times\triangle)$}

In this case, $\cP$ can be viewed as a pyramid in two different ways with
different apexes and bases.  The two base facets are of the type
$\Pyramid(\triangle\times\triangle)$, and are represented by two vertices $u$
and $v$ in the Coxeter graph.  The rest of the graph consists of two parts, say
$G_1$ and $G_2$, representing the two simplices.  The intersection of the two
base facets is a ridge of $\cP$ with the combinatorial type
$\triangle\times\triangle$.  Vertices on this base ridge are all simple.  Every
$k$-face of $\cP$, except for the two apexes and the edge connecting them,
corresponds to a subgraph of $G$ obtained by deleting $k+2$ vertices, including
at least one vertex from both $G_1$ and $G_2$.  The stabilizer of the edge
connecting the two apexes is represented by $G_1+G_2$.  Its level is $0$ since
it corresponds to an edge of $\cP$. 

If a $(2,0)$-graph $H$ has only two real vertices $u$ and $v$, and $H-u-v$ is
of affine type, then we say that $u+v$ is the hinge of $H$.
\begin{lemma}
	If $\cP$ has the combinatorial type of a $2$-fold pyramid over a product of
	two simplices, then
	\begin{enumerate}[label=\emph{(\roman*)}]
		\item $G_1$ and $G_2$ are both affine and are not connected to each
			other;

		\item $G_1+u+v$ and $G_2+u+v$ are both $(2,0)$-graphs with hinge $u+v$;

		\item for any $v_2\in G_2$, the subgraph $G_1+u+v+v_2$ is a $(3,0)$-graph,
			in which $v_2$ is a surreal vertex, while no vertex of $G_1$ is surreal.
	\end{enumerate}
\end{lemma}

The proof use the same type of arguments as before, and use the fact that
$G_1+G_2+u$ and $G_1+G_2+v$ are $(1,1)$-pyramids enumerated
in~\cite{tumarkin2004}.  We now sketch the procedure for enumerating Coxeter
polytopes of this type.

If one of the simplices is of dimension $1$, we construct a candidate
$(2,1)$-graph as follows.  Let $H$ be an affine graph and $H+u+v$ be a
$(2,0)$-graph with hinge $u+v$.  We extend $H$ to a $(3,0)$-graph $H+w$ such
that $w$ is a surreal vertex but no vertex in $H$ is surreal.  We extend $H$ to
another $(3,0)$-graph $H+w'$ in a second (possibly the same) way, and connect
$w$ and $w'$ by a solid edge with label $\infty$.  We require further that
$u+v+w$ and $u+v+w'$ are of level $0$, and $u+w+w'$ and $v+w+w'$ are connected.
This guarantees that $u+v+w+w'$ is a $(2,0)$-graph and $u+v$ is the hinge.
What we obtain is then a candidate $(2,1)$-graph.  

All $(2,0)$-graphs with a hinge and $\ge 5$ vertices are listed in
Figure~\ref{fig:L2Hinges}.  Based on this list, the procedure above gives $221$
candidate $(2,1)$-graphs.  After verification of corank and level, $49$ of them
are confirmed as $(2,1)$-graphs.  They are listed in table~\ref{tbl:Pyr2-1n},
in which we give the position of the $(2,0)$-graph $H+u+v$ in
Figure~\ref{fig:L2Hinges}, and the four labels on the edges connecting $w$ and
$w'$ to $u$ and $v$.

If both simplices are of dimension $1$, the Coxeter graph of the
$4$-dimensional $2$-fold pyramid over a square is in the form of
\tikz[scale=.8, baseline=-.5em]{
  \fill (0,.5) circle (.1);
  \fill (0,-.5) circle (.1);
  \fill (150:1) circle (.1);
  \fill (210:1) circle (.1);
  \fill (30:1) circle (.1);
  \fill (-30:1) circle (.1);
  \draw (0,.5)--(150:1)--(0,-.5)--(210:1)
  		--(0,.5)--(30:1)--(0,-.5)--(-30:1)
  		--(0,.5)--(0,-.5);
	\draw (150:1) -- node[midway] {$\infty$} (210:1); 
	\draw (30:1) -- node[midway] {$\infty$} (-30:1); 
}.
To be a $(2,1)$-graph, each of the four unlabeled triangles should be of level
$0$, either triangle on the left and either triangle on the right should form a
graph of level~$\le 1$, and the corank should be $1$.   We do not have a
complete characterisation for this case.

If both simplices are of dimension $>1$, we construct a candidate $(2,1)$-graph
by taking two $(2,0)$-graphs with hinges and identifying their hinges (possibly
in two different ways).  We then verify the corank and the level of each
candidate.  The latter is done by checking the level of $G-v_1-v_2$ for each
$v_1\in G_1$ and $v_2\in G_2$.  Recall that $G$ is of level $2$, if the level
of $G-v_1-v_2$ is always $\le 1$ but not always $0$.  In
Table~\ref{tbl:Pyr2-nn}, we list all the $36$ polytopes of this class by giving
the position of $G_1+u+v$ and $G_2+u+v$ in Figure~\ref{fig:L2Hinges}
respectively.  It turns out that, for every pair in the table, there is a
unique way to identify the hinges up to graph isomorphism.

\subsection{Remark and discussion}

We have seen a lot of level-$2$ Coxeter graphs.  The algorithms for
classification are implemented in the computer algebra system Sage~\cite{sage}.
For some cases of low rank, because of the large number (even infinite) of
graphs, we gave characterisations and construction methods instead of explicit
lists.  For pyramid (space-like apex) and $2$-fold pyramids over squares, our
characterisation is not satisfactory.  For pyramids over triangular prisms with
space-like apex, we ruled out labels of large value by computer program, but
the reliability of computer can be questioned.

All these graphs correspond to infinite ball packings that are generated by
inversions.  For explicit images of ball packings, the readers are referred to
the artworks of Leys'~\cite{leys2005}.  The $3$-dimensional ball packings in
Leys' paper (and also on his website) are inspired from~\cite{baram2004}.
Similar idea was also proposed by Bullett and Mantica~\cite{bullett1992,
bullett1995}, who also noticed generalizations in higher dimensions.

However, the packings considered in these literatures are very limited.  In our
language, the Coxeter polytopes associated to these packings only have the
combinatorial type of pyramid over regular polytopes.  In~\cite{bullett1992},
the authors were aware of Maxwell's work, but explained that: 
\begin{quote}
	Our approach via limit sets of Kleinian groups is more naive, replacing
	arguments about weight vectors in Minkowsky N-space by elementary geometric
	arguments involving polygonal tiles on the Poincare disc: it mirrors the
	algorithm we use to construct the circle-packings and seems well adapted to
	computation of the exponent of the packing and other scaling constants.
\end{quote}
On the contrary, weight vectors are very useful for investigations.  In fact,
weight vectors only make the computation of the \emph{exponent} (the growth
rate of the curvatures) much easier.  One easily verifies that the height of a
weights is asymptotically equivalent to the curvature of the corresponding
ball.

\begin{remark}
	The Hausdorff dimension of the residual set of infinite ball packings are
	usually approximated by computing the exponent.  In the literature, Boyd's
	works (e.g.~\cite{boyd1974}) are often cited to support this numeric
	estimate.  However, this was not fully justified until recently by Oh and
	Shah~\cite{oh2012}.
\end{remark}

Allcock~\cite{allcock2006} proved that there are infinitely many Coxeter
polytopes in lower dimensional hyperbolic space.  However, we would like to
point out that the situation is not completely dark.  We notice that the
``doubling trick'' used in Allcock's construction produces Coxeter subgroups of
finite index, so the infinitely many hyperbolic Coxeter groups constructed
in~\cite{allcock2006} are all commensurable.  It has been noticed
in~\cite{maxwell1982}*{\S~4} that commensurable Coxeter groups of level-$2$
correspond to the same ball packing.  Indeed, if two Coxeter groups are
commensurable, their Coxeter complex is the subdivision of the same coarser
Coxeter complex.  

Therefore, it makes more sense to enumerate commensurable classes of Coxeter
groups, as Maxwell did in \cite{maxwell1982}*{Table~II}.  For root systems
of corank $0$, the commensurable classes and subgroup relations have been
studied for level~$1$ and $2$ in~\cite{maxwell1998}, and are completely
determined for level~$1$ by Johnson et al.~\cite{johnson2002}.  Despite of
Allcock's result, we may still ask: Are there infinitely many commensurable
classes for level-$l$ Coxeter groups acting on lower dimensional hyperbolic
spaces?  For level $1$ Coxeter groups, the answer is ``yes'' in dimension $2$
(triangle groups), $3$~\cite{maclachlan2003}*{\S~4.7.3}, $4$ and $5$
\citelist{\cite{makarov1968} \cite{vinberg1985}*{\S~5.4}}.  The constructions
in dimension $3$--$5$ made use of level-$1$ polytopes of low corank.

\section*{Acknowledgement} 

The author is grateful to Pavel Tumarkin for inspiring discussions and
interesting references during my one-day visit at Durham university, which
helped improving the paper and the program.  I would also like to thank
Jean-Philippe Labb\'e for helpful discussions, as well as Michael Kapovich and
the anonymous referee for suggestions on preliminary versions of the paper.  

\bibliography{References}

{\footnotesize
	\begin{table}[p]
		\begin{tabular}{c c l|c c l}
			$G_1$ & $G_2$ & \multicolumn{1}{l|}{Edges between $G_1$ and $G_2$}&
			$G_1$ & $G_2$ & \multicolumn{1}{l}{Edges between $G_1$ and $G_2$}\\
			 2& 13& (1,0,3), (2,2,3)&
			 4& 19& (0,0,3), (3,1,3)\\
 			 6& 11& (0,1,3)&
 			 6& 13& (0,0,3), (0,2,3)\\
 			10& 17& (0,1,3), (3,1,3)&
 			10& 22& (0,0,3), (3,1,3)\\
 			11& 14& (1,0,3)&
 			11& 22& (1,2,3)\\
 			12& 12& (0,1,3)&
 			12& 12& (2,2,3)\\
 			12& 15& (0,0,3), (0,2,3)&
 			12& 15& (2,1,3)\\
 			12& 19& (0,0,3), (0,1,3)&
 			12& 19& (1,2,3)\\
 			12& 24& (2,0,3), (2,1,3)&
 			12& 27& (0,0,3), (2,1,3)\\
 			13& 14& (0,0,3), (2,0,3)&
 			13& 22& (0,0,3), (2,1,3)\\
 			13& 26& (0,0,3), (2,2,3)&
 			15& 15& (0,0,3), (2,2,3)\\
 			15& 15& (1,1,3)&
 			15& 19& (0,0,3), (2,1,3)\\
 			16& 30& (0,0,3), (2,1,3)&
 			17& 22& (0,0,3), (2,1,3)\\
 			17& 26& (0,0,3), (2,2,3)&
 			18& 18& (0,1,3), (1,0,3)\\
 			24& 24& (0,0,3), (1,1,3), (2,2,3)&
 			25& 25& (0,0,3), (0,0,4), (1,1,3), (1,1,4)\\
		\end{tabular}
		\caption{
			The first two columns are the positions of $G_1$ and $G_2$ in
			Figure~\ref{fig:L1Ports}, and the third columns are the edges connecting
			$G_1$ and $G_2$.  The ports in Figure~\ref{fig:L1Ports} are numbered, so
			the edges are represented in the format of (port in $G_1$, port in $G_2$,
			label).  By connecting $G_1$ and $G_2$ by the indicated edges, we obtain
			the $(2,1)$-graphs for the products of two simplices (both of dimension
			$>1$).  
		}
 		\label{tbl:Pyr0-nn}
	\end{table}

	\begin{table}[p]
		\centering
		\begin{tabular}{r@{--}l r@{--}l r@{--}l r@{--}l r@{--}l r@{--}l r@{--}l r@{--}l r@{--}l r@{--}l}
			1 &2 & 1 &5 & 1 &9 & 1 &12& 1 &15& 1 &16& 1 &19& 1 &22& 1 &24& 1 &27\\
			2 &3 & 2 &7 & 2 &8 & 2 &23& 3 &5 & 3 &13& 3 &15& 3 &16& 3 &20& 3 &28\\
			4 &9 & 4 &12& 4 &19& 4 &22& 4 &24& 4 &27& 5 &7 & 5 &8 & 5 &23& 7 &9\\
			7 &12& 7 &15& 7 &16& 7 &19& 7 &22& 7 &24& 7 &27& 8 &9 & 8 &12& 8 &15\\
			8 &16& 8 &19& 8 &22& 8 &24& 8 &27& 9 &26& 10&12& 10&19& 10&27& 11&12\\
			11&19& 11&27& 12&18& 12&26& 13&23& 15&23& 16&23& 18&19& 18&27& 19&26\\
			20&23& 22&26& 23&28& 24&26& 26&27
		\end{tabular}
		\caption{For each pair $i$--$j$ in the list, by identifying the white vertices of the $i$-th and the $j$-th graph in Figure~\ref{fig:L1Hinges}, we obtain the $(2,1)$-graph of a pyramid over the product of two simplices (both of dimension $>1$).}
 		\label{tbl:Pyr1-0-nn}
	\end{table}

	\begin{table}[p]
		\centering
		\begin{tabular}{r@{:}l r@{:}l r@{:}l r@{:}l r@{:}l r@{:}l r@{:}l}
 	 	 	 4&(2,3)(3,2)&  4&(2,4)(4,2)&  5&(2,3)(3,2)&  6&(2,3)(3,2)& 15&(2,2)(3,3)\\
 	 	 	 15&(2,3)(3,2)& 15&(2,4)(4,2)& 15&(3,3)(3,3)& 15&(3,4)(4,3)& 24&(2,3)(4,3)\\
 	 	 	 28&(2,2)(3,3)& 28&(2,3)(3,2)& 28&(3,3)(3,3)& 32&(2,2)(3,3)& 32&(2,3)(3,2)\\
 	 	 	 32&(2,4)(4,2)& 32&(3,3)(3,3)& 37&(2,3)(3,2)& 38&(2,2)(3,3)& 38&(2,3)(3,2)\\
 	 	 	 38&(3,3)(3,3)& 39&(2,3)(3,2)& 40&(2,2)(3,3)& 40&(2,3)(3,2)& 40&(3,3)(3,3)\\
 	 	 	 41&(2,2)(3,3)& 41&(3,3)(3,3)& 42&(2,2)(3,3)& 42&(3,3)(3,3)& 48&(3,2)(3,3)\\
 	 	 	 49&(2,2)(4,3)& 49&(2,3)(4,2)& 49&(3,2)(3,4)& 49&(4,3)(4,3)& 57&(2,3)(4,3)\\
			 59&(2,2)(3,3)& 59&(3,3)(3,3)& 61&(2,2)(3,3)& 61&(2,3)(3,2)& 61&(2,4)(4,2)\\
			 61&(3,3)(3,3)& 65&(2,3)(3,2)& 66&(2,2)(3,3)& 66&(2,3)(3,2)& 66&(3,3)(3,3)\\
			 67&(2,3)(3,2)& 68&(2,2)(3,3)& 68&(2,3)(3,2)& 68&(3,3)(3,3)
		\end{tabular}
		\caption{
			For each entry $i$:$(a,b)(c,d)$ in the list, take the $i$-th graph
			$H+u+v$ in Figure~\ref{fig:L2Hinges}, where $u$ is the gray vertex and
			$v$ is the white vertex.  Introduce two new vertices $w$ and $w'$, and
			connect them to $H$ such that $wu$ has label $a$, $wv$ has label $b$,
			$w'u$ has label $c$, $w'v$ has label $d$, and finally label the edge
			$ww'$ by $\infty$.  The result is the $(2,1)$-graph of a $2$-fold pyramid
			over a prism.
		}
 		\label{tbl:Pyr2-1n}
	\end{table}

	\begin{table}[p]
		\centering
		\begin{tabular}{r@{--}l r@{--}l r@{--}l r@{--}l r@{--}l r@{--}l r@{--}l r@{--}l r@{--}l r@{--}l r@{--}l r@{--}l}
 	 	 	 4&4 &  8&15&  8&22&  8&56&  8&62& 13&13\\
 	 	 	 13&49& 15&15& 15&22& 15&32& 15&35& 15&54\\
			15&56& 15&61& 15&62& 22&22& 22&32& 22&35\\
			22&54& 22&56& 22&61& 22&62& 32&56& 32&62\\
			35&56& 35&62& 38&56& 49&49& 54&56& 54&62\\
			56&56& 56&61& 56&62& 56&66& 61&62& 62&62
		\end{tabular}
		\caption{
			For each pair $i$--$j$ in the list, by identifying the white/light-gray
			vertices of the $i$-th and the $j$-th graph in Figure~\ref{fig:L2Hinges},
			we obtain the $(2,1)$-graph of a $2$-fold pyramid over the product of two
			simplices (both of dimension $>1$).
		}
 		\label{tbl:Pyr2-nn}
	\end{table}
}

\begin{figure}[h]
	\centering
	\includegraphics[height=.9\textheight]{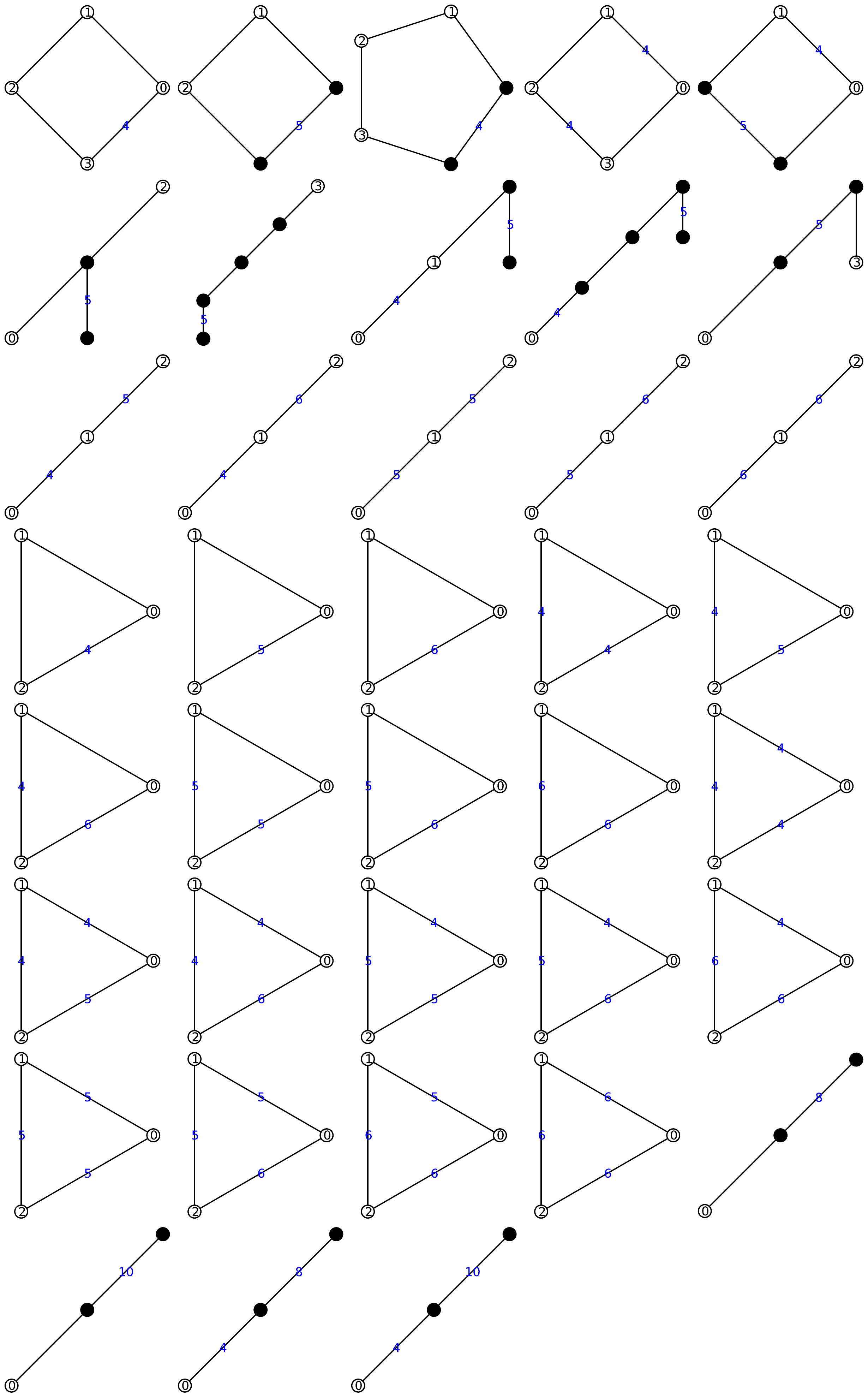}
	\caption{$(1^s,0)$-graphs of $\ge 3$ vertices with ports (numbered white vertices)}
	\label{fig:L1Ports}
\end{figure}

\begin{figure}[p]
	\centering
	\includegraphics[width=\textwidth]{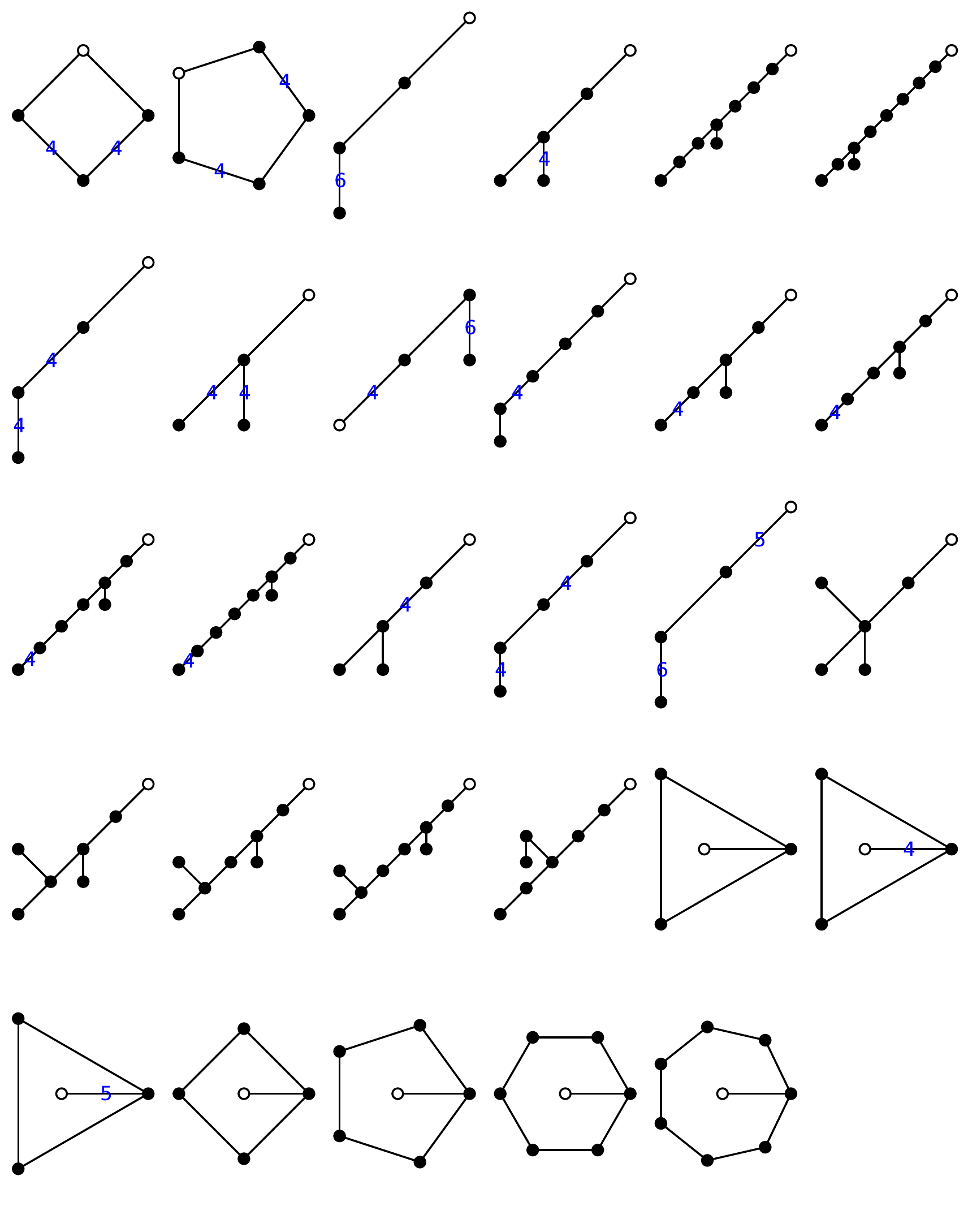}
	\caption{non-strict $(1,0)$-graphs of $\ge 4$ vertices with a hinge (the white vertex)}
	\label{fig:L1Hinges}
\end{figure}

\begin{figure}[p]
	\centering
  \includegraphics[height=.9\textheight]{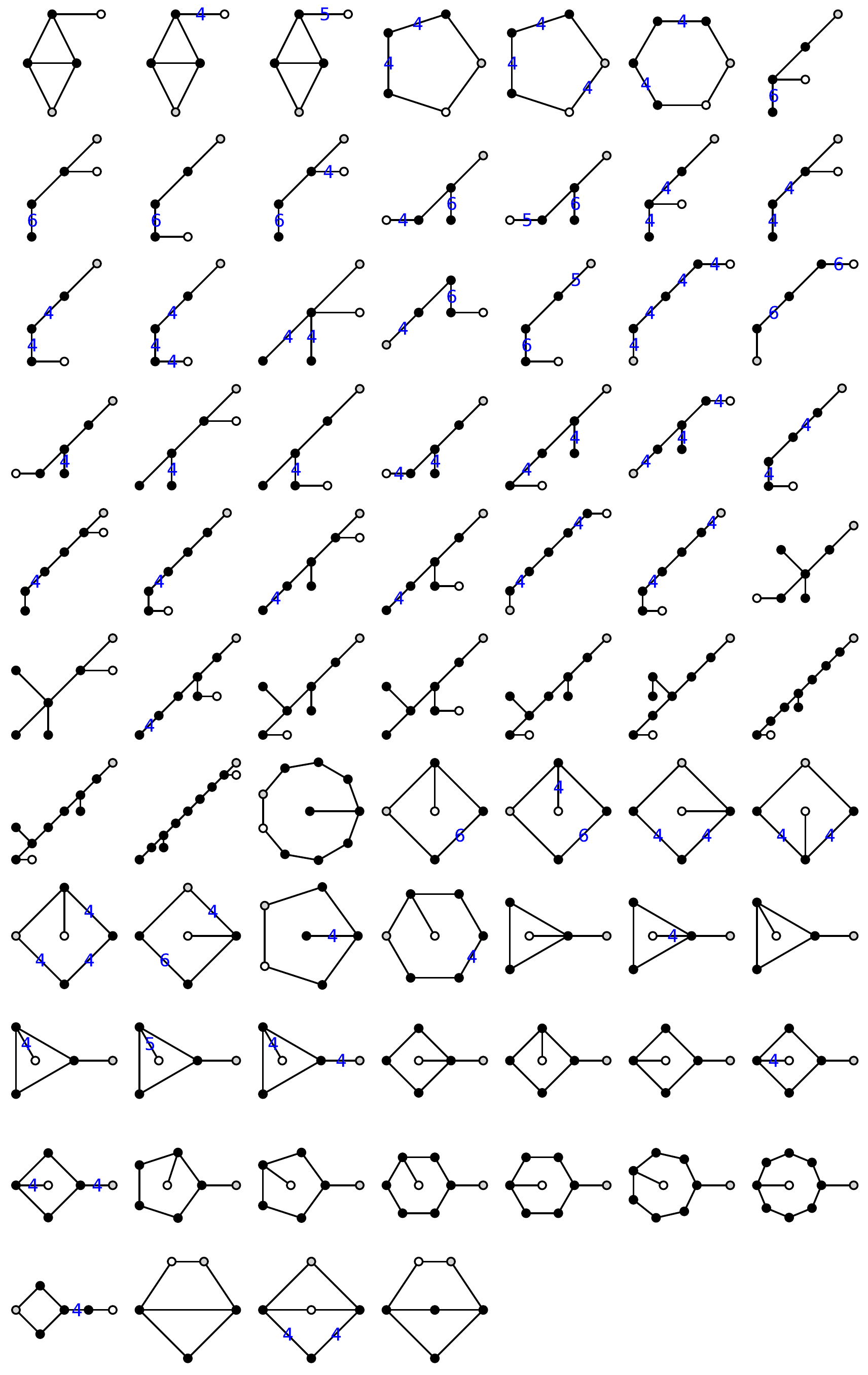}
  \caption{$(2,0)$-graphs of $\ge 5$ vertices with a hinge (the white and the light-gray vertices)}
  \label{fig:L2Hinges}
\end{figure}

\end{document}